\documentclass[12pt,leqno]{amsart}
\usepackage{amssymb}
\usepackage{mathrsfs}
\usepackage{graphicx,color}
\usepackage{newtxtext}
\usepackage[varg]{newtxmath}
\theoremstyle{plain} 
\newtheorem{theorem}{Theorem}[section]
\newtheorem{corollary}[theorem]{Corollary}
\newtheorem{proposition}[theorem]{Proposition}
\newtheorem{lemma}[theorem]{Lemma}

\theoremstyle{definition} 

\newtheorem{remark}[theorem]{Remark}

\newcommand{\R}{\mathbb{R}}

\usepackage{constants}
\numberwithin{equation}{section}
\usepackage{bm}
\renewcommand{\vec}[1]{\bm{#1}}
\usepackage{hyperref}

\newcommand{\NorAngle}{\phi}
\newcommand{\MisOriAngle}{\theta}

\title{Motion of grain boundaries with dynamic lattice misorientations
  and with triple junctions drag}

\author{Yekaterina Epshteyn}
\address[Yekaterina Epshteyn]%
{Department of Mathematics,
The University of Utah,
Salt Lake City, UT 84112, USA}
\email{epshteyn@math.utah.edu}

\author{Chun Liu}
\address[Chun Liu]%
{Department of Applied Mathematics, Illinois Institute of Technology.
Chicago, IL 60616, USA}
\email{cliu124@iit.edu}

\author{Masashi Mizuno}
\address[Masashi Mizuno]%
{Department of Mathematics, College of Science
and Technology, Nihon University, Tokyo 101-8308 JAPAN}
\email{mizuno@math.cst.nihon-u.ac.jp}

\keywords{Grain growth, grain boundary network, texture development,
  lattice misorientation, triple junction drag, energetic
variational approach, geometric evolution equations}
\subjclass[2000]{74N15, 35R37, 53C44, 49Q20}

\allowdisplaybreaks[1]
\begin{document}

%
%


%
%
%

\begin{abstract}
Most technologically useful materials are polycrystalline
microstructures composed of a myriad of small monocrystalline grains
separated by grain boundaries. The  energetics and connectivities of grain boundaries play a
crucial role in defining the main characteristics of materials across a
wide range of scales. In this work, we propose a model for the
evolution of the grain boundary network with dynamic boundary conditions at the
triple junctions, triple junctions drag, and with dynamic lattice misorientations. Using the energetic
variational approach, we derive system of
geometric differential equations to describe motion of such
grain boundaries. Next, we relax curvature effect of the grain
boundaries to isolate the effect of the dynamics of lattice
misorientations and triple junctions drag, and we establish local
well-posedness result for the considered model.
\end{abstract}

\maketitle

\section{Introduction}
\par  Most technologically useful materials are polycrystalline
microstructures composed of a myriad of small monocrystalline grains
separated by grain boundaries. The  energetics and connectivities of grain boundaries play a
crucial role in defining the main characteristics of materials across a
wide range of scales.   More recent mesoscale experiments and simulations provide large amounts of information
about both geometric features and crystallography of the grain boundary network in material
microstructures.
\par For the time being, we will focus on a planar grain boundary network. A
classical model,  due to Mullins and Herring~\cite{doi:10.1007-978-3-642-59938-5_2,
doi:10.1063-1.1722511,doi:10.1063-1.1722742},  for the evolution of grain
boundaries in polycrystalline materials is based on the motion by mean curvature as the local
evolution law. Under the assumption that the total
grain boundary energy depends only on the surface tension of the
grain boundaries, the motion by mean curvature is consistent
with the dissipation principle for the total grain boundary energy.  In
addition, to have a well-posed model of the evolution of the grain
boundary network, one has to impose a separate condition at the triple
junctions where three grain boundaries meet \cite{MR1833000}.
Note, that at equilibrium state, the energy is minimized, which
implies that a force balance, known as the Herring Condition,
holds at the triple junctions. Herring
condition is the natural boundary condition for the system at the
equilibrium.
However, during the evolution of the grain boundaries, the normal
velocity of the boundary is proportional to a driving force. Therefore, unlike
the equilibrium state, there is no natural boundary condition for an
evolutionary system, and one must be stated. A standard choice is the
Herring condition  \cite{MR0485012,MR1240580,MR1833000,MR3612327}, and reference
therein. There are several mathematical studies
about the motion by mean curvature of grain boundaries with the Herring
condition at the triple junctions, see for example
\cite{MR1833000,MR3495423,MR2076269,MR3565976,arXiv:1611.08254,MR2075985,
  DK:BEEEKT,DK:gbphysrev, MR2772123, MR3729587, BobKohn, barmak_grain_2013, MR3316603}. There are
some computational studies too
\cite{doi:10.1023-A:1008781611991,doi:10.1016-S1359-6454(01)00446-3, MR2772123, MR2748098, MR2573343,MR3787390,MR3729587}.
\par A basic assumption in the theory and simulations of the grain
growth is the motion of the grain boundaries themselves and not the
motion of the triple junctions. However, recent experimental studies
indicate that the motion of triple junctions together with anisotropy of
the grain boundary network can have an important effect on the grain
growth \cite{barmak_grain_2013}, and see also a recent work on
dynamics of line
defects \cite{Zhang2017PhysRevLett, Zhang2018JMPS,Thomas2019PNAS}. In
this work, to investigate the evolution of the anisotropic network of
grain boundaries, we propose a new model that assumes that
interfacial/grain boundary energy density is a function of dynamic
lattice misorientations. Moreover, we impose a dynamic boundary
condition at the triple junctions, a triple junctions drag.  The
proposed model can be viewed as a multiscale model containing the local
and long-range interactions of the lattice misorientations and the
interactions of the triple junctions of the grain boundaries.  Using the
energetic variational approach, we derive the system of geometric
differential equations to describe the motion of such grain
boundaries.
%
%
Next, we relax the curvature effect of the grain boundaries to isolate
the effect of the dynamics of lattice misorientations and triple
junctions drag, and we establish local well-posedness result for the
considered model.  Note that, the current work is motivated and closely
related to the work \cite{MR1833000} (where well-posedness of the grain
boundary network model with Herring condition at the triple junctions
and with no misorientation effect was established), and to the work
\cite{DK:gbphysrev,MR2772123,MR3729587} (where a reduced 1D model based
on the dynamical system was studied for texture evolution and was used
to identify texture evolution as a gradient flow).

\par The paper is organized as follows. In
Section~\ref{sec:2} we derive a new model for the grain boundaries. In
Sections~\ref{sec:3}-\ref{sec:6} we show local well-posedness of the
proposed model under assumption of a single triple junction. Finally, in
Section~\ref{sec:7}, we extend the obtained results for a system with a
single triple junction to the grain boundary network with multiple
junctions.

\section{Derivation of the model}
\label{sec:2}
In this section we present the derivation of the model with dynamic
lattice misorientations and with triple junctions drag. This is further
extension of the model in \cite{MR1833000}, and it is motivated by the work in \cite{DK:gbphysrev,MR2772123,MR3729587}.
 
 First, we obtain our model for the evolution of the grain boundaries
using energy dissipation principle for the system. Note, while
critical events (such as, disappearance of the grains and/or grain boundaries
during coarsening of the system) pose a great challenge on the
modeling, simulation and analysis, see Fig. \ref{fig3}, here we start
with a system of one triple junction to obtain a consistent model, see
Fig. \ref{fig:1}. Thus, we start the derivation by considering
the system of three curves only, that meet at a single point -- a triple
junction $\vec{a}(t)$, see Fig. \ref{fig:1}:
\[
 \Gamma_t^{(j)}:\vec{\xi}^{(j)}(s,t),\quad
 0\leq s\leq 1,\quad
 t>0,\quad
 j=1,2,3.
\]
These curves satisfy the following conditions at the triple junction
and at the end points of the curves,
\[
 \vec{a}(t)
 :=
 \vec{\xi}^{(1)}(0,t)
 =
 \vec{\xi}^{(2)}(0,t)
 =
 \vec{\xi}^{(3)}(0,t),
 \quad
 \text{and}
 \quad
 \vec{\xi}^{(j)}(1,t)=\vec{x}^{(j)},\quad
 j=1,2,3.
\]
Here, we assume that curves $\Gamma_t^{(j)}, j=1, 2, 3$ are
sufficiently smooth functions of  parameter  $s$ (not necessarily the
arc length) and time $t$. Also, for now we assume that endpoints of
the curves $\vec{x}^{(j)}\in\R^2$ are fixed points, see Fig. \ref{fig:1}.  We define a
tangent vector $\vec{b}^{(j)}=\vec{\xi}^{(j)}_s$
and a normal vector
$\vec{n}^{(j)}=R\vec{b}^{(j)}$ (not necessarily the unit vectors) to each curve,
where $R$ is the rotation matrix through $\pi/2$. We denote
$\Gamma_t:=\Gamma_t^{(1)}\cup\Gamma_t^{(2)}\cup\Gamma_t^{(3)}$. We
also consider below a standard euclidean vector norm denoted $|\cdot |$.
\begin{figure}[hbtp]
\centering
\includegraphics[width=3.0in]{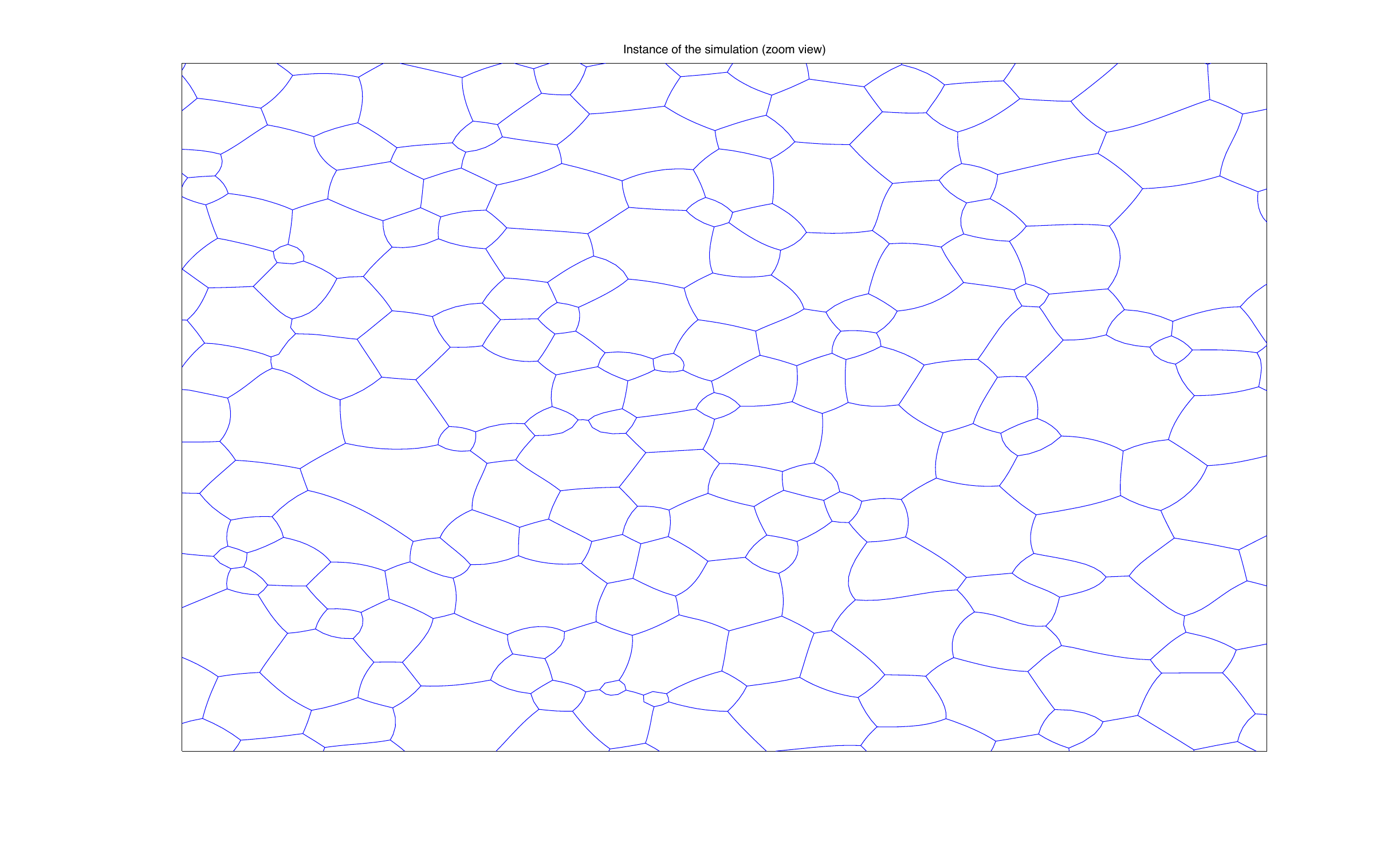}\quad

\caption{ Time instance from the simulation of the 2D grain
  boundary network with dynamic/time-dependent orientation (zoom view). }\label{fig3}
\end{figure}

Now, for $j=1,2,3$, let $\alpha^{(j)}=\alpha^{(j)}(t)$ be the lattice
orientation angle of the grain which is enclosed between grain
boundaries $\Gamma_t^{(j)}$ and $\Gamma_t^{(j+1)}$, and we set that
$\Gamma_t^{(4)}=\Gamma_t^{(1)}$ for the simplicity of the
notation. Similar to work
\cite{DK:BEEEKT,DK:gbphysrev,MR2772123,MR3729587,MR3333842}, we assume
here that the orientation $\alpha^{(j)}$ is a bounded scalar since we
consider a planar grain boundary network. In this work, we make an
assumption that lattice orientations are functions of time $t$ (we
assume that during grain growth, grains can change their lattice
orientations due to rotation), but independent of the parameter
$s$. Next, we define, the surface energy density or interfacial grain
boundary energy of $\Gamma_t^{(j)}$ as
\[
 \sigma=\sigma(\vec{n}^{(j)},\alpha^{(j-1)}-\alpha^{(j)})=\sigma(\vec{n}^{(j)},\Delta \alpha^{(j)})\geq0,
\]
where we denote  $\Delta\alpha^{(j)}:=\alpha^{(j-1)}-\alpha^{(j)}$ to
be misorientation angle across the grain boundary (a common boundary
for  two neighboring grains with orientations $\alpha^{(j-1)}$ and $\alpha^{(j)}$), and we set for
convenience
$\alpha^{(0)}:=\alpha^{(3)}$, see Fig. \ref{fig:1}. See also Remark
\ref{rem:5.5} in Section \ref{sec:5}.
\par The
total grain boundary energy of the system $\Gamma_t$ can be obtained as
\begin{equation}\label{eq:1}
 E(t)
 =
 \sum_{j=1}^3
 \int_{\Gamma_t^{(j)}}\sigma(\vec{n}^{(j)},\Delta\alpha^{(j)})
 \,d\mathscr{H}^1
 =
 \sum_{j=1}^3
 \int_{0}^{1}
 \sigma(\vec{n}^{(j)},\Delta\alpha^{(j)})|\vec{b}^{(j)}|\,ds,
\end{equation}
where $\mathscr{H}^1$ is the $1$-dimensional Hausdorff measure, (see
Fig.~\ref{fig:1}).
\begin{figure}
 \centering
 \includegraphics[height=7cm]{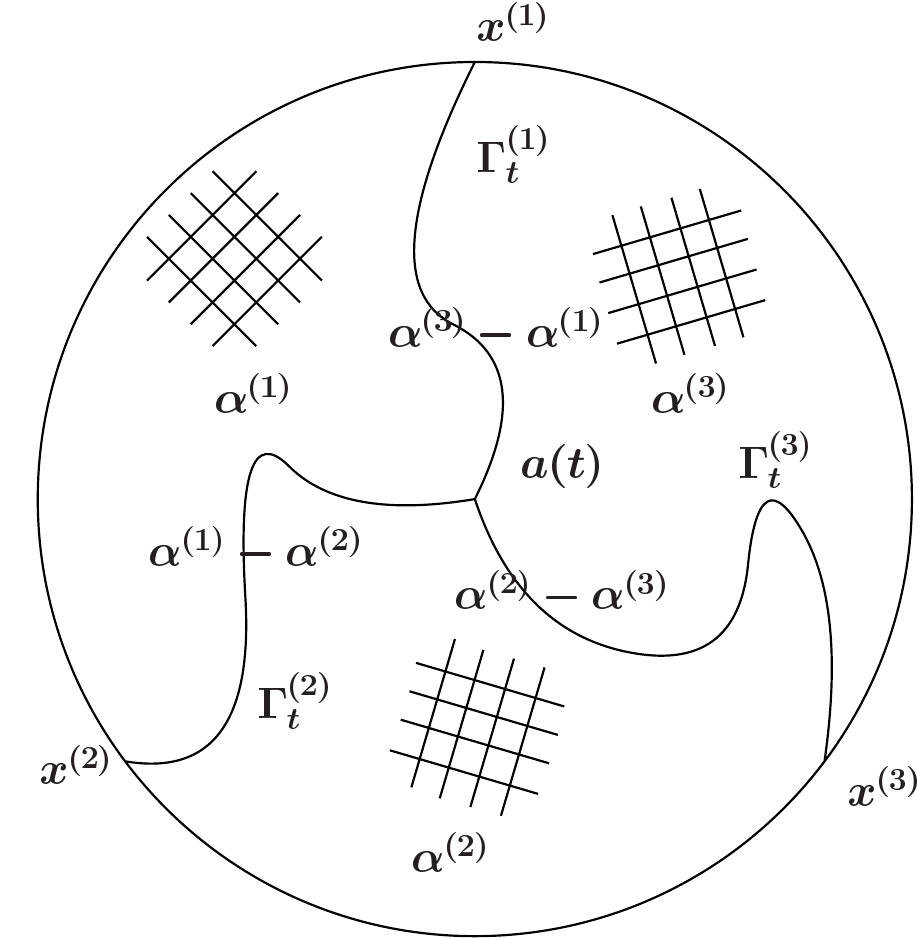}
 \caption{The model of grain boundaries/curves
 $\Gamma^{(j)}_t$ with triple junction  $\vec{a}(t)$ and with
 orientations angles (scalars) $\alpha^{(j)}$.}
 \label{fig:1}
\end{figure}
Next, we use the coordinate $(\vec{n},\MisOriAngle)\in\R^2\times\R$ for
the surface energy density $\sigma (\vec{n},\MisOriAngle)$ and assume
that $\sigma$ is taken to be positively homogeneous of degree $0$ in
$\vec{n}$.  Note, that in general, grain boundaries are identified by
lattice misorientation and the orientation of the normal vector to the
grain boundary. For simplicity of notations, we denote
$\sigma^{(j)}:=\sigma(\vec{n}^{(j)},\Delta\alpha^{(j)})$.

Let us now define grain boundary motion that will result in the
dissipation of the total grain boundary energy (\ref{eq:1}).  Denote
by $\hat{\phantom{n}}$ the normalization operator of vectors,
e.g. $\hat{\vec{n}}^{(j)}=\frac{\vec{n}^{(j)}}{|\vec{n}^{(j)}|}$. Then,
we can compute the rate of change in energy at time $t$ due to grain
boundary motion as follows:
\begin{equation}\label{eq:2}
  \begin{split}
   \frac{d}{dt}E(t)
   &=
   \sum_{j=1}^3
   \biggl(
   \int_{0}^{1}
   \nabla_{\vec{n}}\sigma^{(j)}
   \cdot\frac{d\vec{n}^{(j)}}{dt}|\vec{b}^{(j)}|\,ds
   +
   \int_{0}^{1}\sigma^{(j)}
   \frac{\vec{b}^{(j)}}{|\vec{b}^{(j)}|}
   \cdot\frac{d\vec{b}^{(j)}}{dt}\,ds \\
   &\quad +
   \int_{0}^{1}\sigma^{(j)}_\MisOriAngle
   \frac{d(\Delta \alpha^{(j)}))}{dt}|\vec{b}^{(j)}|\,ds
   \biggr) \\
   &=
   \sum_{j=1}^3
   \biggl(
   \int_{0}^{1}
   \left(
   |\vec{b}^{(j)}|
   \mathstrut^t\!R\nabla_{\vec{n}}\sigma^{(j)}
   +
   \sigma^{(j)}
   \hat{\vec{b}}^{(j)}
   \right)
   \cdot\frac{d\vec{b}^{(j)}}{dt}\,ds \\
   &\quad +
   \int_{0}^{1}
   \sigma^{(j)}_\MisOriAngle
   \frac{d(\Delta \alpha^{(j)}))}{dt}|\vec{b}^{(j)}|\,ds
   \biggr).
  \end{split}
\end{equation}
Next, consider a polar angle $\NorAngle^{(j)}$ and set
$\hat{\vec{n}}^{(j)}=(\cos\NorAngle^{(j)},\sin\NorAngle^{(j)})$.  Since
$\sigma^{(j)}$ is positively homogeneous of degree $0$ in
$\vec{n}^{(j)}$, we have
\begin{equation*}
 \nabla_{\vec{n}}\sigma\cdot\vec{n}=0,
  \quad
  \mathstrut^tR\nabla_{\vec{n}}\sigma
  =
  (\mathstrut{}^tR\nabla_{\vec{n}}\sigma\cdot\hat{\vec{n}})\hat{\vec{n}}
  ,\quad
  \sigma^{(j)}_\NorAngle
  \hat{\vec{n}}^{(j)}
  =
  |\vec{b}^{(j)}|
  \mathstrut^t\!R
  \nabla_{\vec{n}}\sigma^{(j)},
\end{equation*}
and, thus, we define the vector $\vec{T}^{(j)}$ known as the line tension or capillary stress vector,
\begin{equation*}
 \vec{T}^{(j)}
  :=\sigma^{(j)}_\NorAngle
  \hat{\vec{n}}^{(j)}
  +
  \sigma^{(j)}
  \hat{\vec{b}}^{(j)}
  =
  |\vec{b}^{(j)}|
   \mathstrut^t\!R\nabla_{\vec{n}}\sigma^{(j)}
   +
   \sigma^{(j)}
   \hat{\vec{b}}^{(j)}.
\end{equation*}
Now, using the change of variable
\[
 \frac{d\vec{b}^{(j)}}{dt}
 =
 \frac{d}{ds}
 \frac{d\vec{\xi}^{(j)}}{dt},
\]
we can rewrite (\ref{eq:2}) as:
\begin{equation}\label{eq:2a}
 \begin{split}
  \frac{d}{dt}E(t)
  &=
  \sum_{j=1}^3
  \biggl(
  \int_{0}^{1}
  \vec{T}^{(j)}
  \cdot
  \frac{d}{ds}
  \frac{d\vec{\xi}^{(j)}}{dt}\,ds
  +
  \int_{0}^{1}
  \sigma^{(j)}_\MisOriAngle
  \frac{d(\Delta \alpha^{(j)}))}{dt}|\vec{b}^{(j)}|\,ds
  \biggr) \\
  &=
  -\sum_{j=1}^3
  \int_{0}^{1}
  \vec{T}^{(j)}_s
  \cdot
  \frac{d\vec{\xi}^{(j)}}{dt}\,ds
  +
  \sum_{j=1}^3
  \int_{0}^{1}
  \sigma^{(j)}_\MisOriAngle
  \frac{d(\Delta \alpha^{(j)}))}{dt}|\vec{b}^{(j)}|\,ds
  \\
  &\quad
  -\sum_{j=1}^3
  \vec{T}^{(j)}(0,t)
  \cdot
  \frac{d\vec{a}}{dt}(t).
 \end{split}
\end{equation}
For the reader's convenience, we will recall below the following property for
a divergence of the capillary stress vector $\vec{T}^{(j)}.$\\
\begin{lemma}
 Let $\kappa^{(j)}$ is the curvature of $\Gamma_t^{(j)}$. Then
 \begin{equation}
  \label{eq:2.14}
   \vec{T}^{(j)}_s
   =
   |\vec{b}^{(j)}|
   (\sigma_{\NorAngle\NorAngle}^{(j)}+\sigma^{(j)})\kappa^{(j)}\hat{\vec{n}}^{(j)}.
 \end{equation}
\end{lemma}

\begin{proof}
 From the Frenet-Serret formula for the non-arc length parameter,
 \begin{equation}
  \label{eq:2.10}
   \hat{\vec{b}}_s^{(j)}
   =
   |\vec{b}^{(j)}|
   \kappa^{(j)}
   \hat{\vec{n}}^{(j)},\quad
   \hat{\vec{n}}_s^{(j)}
   =
   -|\vec{b}^{(j)}|
   \kappa^{(j)}
   \hat{\vec{b}}^{(j)}.
 \end{equation}
 Thus we obtain,
 \begin{equation}
  \label{eq:2.11}
   \begin{split}
    \vec{T}_s^{(j)}
    &=
    \left(
    \nabla_{\vec{n}}\sigma^{(j)}_\NorAngle\cdot
    \vec{n}^{(j)}_s
    \right)
    \hat{\vec{n}}^{(j)}
    +
    \sigma^{(j)}_\NorAngle
    \hat{\vec{n}}^{(j)}_s
    +
    \left(
    \nabla_{\vec{n}}\sigma^{(j)}
    \cdot
    \vec{n}^{(j)}_s
    \right)
    \hat{\vec{b}}^{(j)}
    +
    \sigma^{(j)}
    \hat{\vec{b}}^{(j)}_s \\
    &=
    \left(
    \mathstrut^t\!R\nabla_{\vec{n}}\sigma^{(j)}_\NorAngle
    \cdot
    \vec{b}^{(j)}_s
    +
    |\vec{b}^{(j)}|\sigma^{(j)}\kappa^{(j)}
    \right)
    \hat{\vec{n}}^{(j)}
    +
    \left(
    -|\vec{b}^{(j)}|\sigma^{(j)}_\NorAngle\kappa^{(j)}+
    \mathstrut^t\!R\nabla_{\vec{n}}\sigma^{(j)}
    \cdot
    \vec{b}^{(j)}_s
    \right)
    \hat{\vec{b}}^{(j)}.
   \end{split}
 \end{equation}
 Since $\sigma^{(j)}$ and $\sigma_\NorAngle^{(j)}$ are positively
 homogeneous of degree $0$ in $\vec{n}^{(j)}$, we have,
 \begin{equation}
  \label{eq:2.12}
   \sigma^{(j)}_\NorAngle
   \hat{\vec{n}}^{(j)}
   =
   |\vec{b}^{(j)}|
   \mathstrut^t\!R
   \nabla_{\vec{n}}\sigma^{(j)},\quad
   \sigma^{(j)}_{\NorAngle\NorAngle}
   \hat{\vec{n}}^{(j)}
   =
   |\vec{b}^{(j)}|
   \mathstrut^t\!R
   \nabla_{\vec{n}}\sigma_\NorAngle^{(j)}.
 \end{equation}
 Using the orthogonal relation $\vec{b}^{(j)}\cdot\hat{\vec{n}}^{(j)}=0$
 and the Frenet-Serret formula \eqref{eq:2.10}, we obtain,
 \begin{equation}
  \label{eq:2.13}
   \vec{b}_s^{(j)}\cdot\hat{\vec{n}}^{(j)}
   =
   -\vec{b}^{(j)}\cdot\hat{\vec{n}}^{(j)}_s
   =
   |\vec{b}^{(j)}|^2\kappa^{(j)}.
 \end{equation}
 Plugging \eqref{eq:2.12} and \eqref{eq:2.13} into \eqref{eq:2.11}, we
 derive \eqref{eq:2.14}.

\end{proof}

Next, to ensure that the entire system of grain boundaries is dissipative, i.e.
\begin{equation*}
 \frac{d}{dt}E(t)\leq 0,
\end{equation*}
we impose Mullins theory (curvature driven growth)
\cite{doi:10.1063-1.1722742,Mullins:MS:1963} as the local evolution
law stating that the normal velocity $v_n^{(j)}$ of a grain boundary of $\Gamma_t^{(j)}$
(the rate of growth of area adjacent to the boundary $\Gamma^{(j)}_t$),
is proportional to the line force  $\vec{T}_s^{(j)}$ (to the work done through deforming the curve), through the factor of the mobility $\mu^{(j)}>0:$
\begin{equation}
 \label{eq:2.1}
  v_n^{(j)}\hat{\vec{n}}^{(j)}
  =
  \mu^{(j)}
  \frac{1}{|\vec{b}^{(j)}|}\vec{T}_s^{(j)}
  =
  \mu^{(j)}
  (
  \sigma_{\NorAngle\NorAngle}^{(j)}
  +\sigma^{(j)}
  )
  \kappa^{(j)}\hat{\vec{n}}^{(j)}
  \quad \text{on}
  \ \Gamma_t^{(j)},\quad j=1,2,3.
\end{equation}
Note, that using variation of the energy $E$ with respect to the curve $\vec{\xi}^{(j)}$,
namely,
\begin{equation*}
 v_n^{(j)}\hat{\vec{n}}^{(j)}
  =
  -\mu^{(j)}
  \frac{\delta E}{\delta \vec{\xi}^{(j)}},
\end{equation*}
one can derive the following relation for the line force
$\vec{T}_s^{(j)}$  \cite{MR1833000},
\begin{equation}
 \label{eq:2.1a}
  \mu^{(j)}
  \frac{1}{|\vec{b}^{(j)}|}\vec{T}_s^{(j)}
  =
  \mu^{(j)}
  (
  \sigma_{\NorAngle\NorAngle}^{(j)}
  +\sigma^{(j)}
  )
  \kappa^{(j)}\hat{\vec{n}}^{(j)}
  \quad \text{on}
  \ \Gamma_t^{(j)},\quad j=1,2,3.
\end{equation}
Since
$v_n^{(j)}=\frac{d\vec{\xi}^{(j)}}{dt}\cdot\hat{\vec{n}}^{(j)}$, we
obtain that,
\begin{equation}
 \label{eq:2.8}
 \vec{T}^{(j)}_s
  \cdot
  \frac{d\vec{\xi}^{(j)}}{dt}
  =\frac{1}{\mu^{(j)}}|v_n^{(j)}|^2|\vec{b}^{(j)}|
  \geq0,
\end{equation}
and, thus, the first term on the right-hand side of (\ref{eq:2a}) is
non-positive.
Next, we consider the second term on the right-hand side of (\ref{eq:2a}) which depends on the derivative of lattice misorientation, we
have that (since $\alpha^{(j)}$
is independent of $s$),
\[
  \sum_{j=1}^3
  \int_{0}^{1}
   \sigma^{(j)}_\MisOriAngle
   \frac{d(\Delta \alpha^{(j)}))}{dt}|\vec{b}^{(j)}|\,ds
   =
   \sum_{j=1}^3
   \biggl(
   \int_0^{1}
   \Bigl(
   \sigma^{(j+1)}_\MisOriAngle
   |\vec{b}^{(j+1)}|
   -
   \sigma^{(j)}_\MisOriAngle
   |\vec{b}^{(j)}|
   \Bigr)
   \,ds
   \biggr)
   \frac{d\alpha^{(j)}}{dt},
\]
where we used that $\sigma^{(4)}=\sigma^{(1)}$. To ensure,
$\frac{d}{dt}E(t) \le 0$ in (\ref{eq:2a}), we make an assumption that for
a constant $\gamma>0$, we have the following relation for the rate of
change of the lattice orientations,
\begin{equation}
 \label{eq:2.2}
 \frac{d\alpha^{(j)}}{dt}
  =
  -\gamma
  \biggl(
  \int_0^{1}
  \Bigl(
  \sigma^{(j+1)}_\MisOriAngle
  |\vec{b}^{(j+1)}|
  -
  \sigma^{(j)}_\MisOriAngle
  |\vec{b}^{(j)}|
  \Bigr)
  \,ds
  \biggr),\quad
   j=1,2,3
\end{equation}
since the relation (\ref{eq:2.2}) results in the condition,
\begin{equation}
 \label{eq:2.6}
  \sum_{j=1}^3
  \int_{0}^{1}
   \sigma^{(j)}_\MisOriAngle
   \frac{d(\Delta \alpha^{(j)}))}{dt}|\vec{b}^{(j)}|\,ds
   =
   -\frac1\gamma
   \sum_{j=1}^3
   \left|
    \frac{d\alpha^{(j)}}{dt}
   \right|^2
   \leq0
\end{equation}
on the second term in the right-hand side of (\ref{eq:2a}).
Note, that the proposed relation \eqref{eq:2.2} can also be derived using variation of
the energy $E$ with respect to  lattice orientation $\alpha^{(j)}$, namely,
\begin{equation*}
 \frac{d\alpha^{(j)}}{dt}
  =
  -\gamma
  \frac{\delta E}{\delta \alpha^{(j)}}.
\end{equation*}

\begin{remark}\label{rem:2.2}
 1. As we discussed, the misorientations are defined  using the
 orientations, $\alpha^{(j)}$ as,
 $\Delta\alpha^{(j)}=\alpha^{(j-1)}-\alpha^{(j)}$. Conversely, if the sum of the
 misorientations is zero, namely,
 $\Delta\alpha^{(1)}+\Delta\alpha^{(2)}+\Delta\alpha^{(3)}=0$, then
 the following linear relation,
 \begin{equation*}
  \left\{
   \begin{split}
    \alpha^{(3)}-\alpha^{(1)}
    &=\Delta\alpha^{(1)}, \\
    \alpha^{(1)}-\alpha^{(2)}
    &=\Delta\alpha^{(2)},
    \\
    \alpha^{(2)}-\alpha^{(3)}
    &=\Delta\alpha^{(3)}
   \end{split}
  \right.
 \end{equation*}
 can be solved in terms of $\alpha^{(j)}$, and the (inverse) mapping,
 \begin{equation*}
   \begin{pmatrix}
   \Delta\alpha^{(1)},\
    \Delta\alpha^{(2)},\
    \Delta\alpha^{(3)}
   \end{pmatrix}
   \mapsto
   \begin{pmatrix}
    c-\Delta\alpha^{(1)},\
    c+\Delta\alpha^{(3)},\
    c
   \end{pmatrix}
   =
   \begin{pmatrix}
    \alpha^{(1)},\
    \alpha^{(2)},\
    \alpha^{(3)}
   \end{pmatrix}
 \end{equation*}
 gives the orientations from the misorientations $\Delta\alpha^{(j)}$.
 Here $c$ is an arbitrary parameter. Thus, if we would formulate/derive
 equations for the misorientation evolution, instead of
 the equation for the orientation (\ref{eq:2.2}), we
 would have to impose additional constraint,
 $(\Delta\alpha^{(1)}+\Delta\alpha^{(2)}+\Delta\alpha^{(3)})(0)=0$. Furthermore,
 in that case, the orientation of each grain may not be determined
 uniquely due to the arbitrary parameter $c$.
 On the other hand, from \eqref{eq:2.2} it follows directly that,
 \begin{equation*}
  \frac{d}{dt}(\alpha^{(1)}+\alpha^{(2)}+\alpha^{(3)})=0.
 \end{equation*}
 Hence, the sum of the orientations
 $\alpha^{(1)}+\alpha^{(2)}+\alpha^{(3)}$ has to be a constant. This
 constraint for the orientations is easily determined by the initial
 configuration, and both the orientations and the misorientations can be
 determined from the equation~\eqref{eq:2.2}.\\ 2. As discussed above,
 in our work, we consider the orientation as the primary variable, and
 we enforce dissipation in the system by assuming relation
 (\ref{eq:2.2}) through the orientation. Note that, we consider the rate
 of the change on the orientation (rather than on the misorientation)
 since we study system before critical events/disappearance
 events. Moreover, this choice of the orientation as the primary
 variable is also consistent with a case of grain boundary energy
 $\sigma(\vec{n}^{(j)}, \Delta \alpha^{(j)})$. In addition, note that,
 the traditional texture distribution is the orientation
 distribution. However, in general, one can obtain the misorientation
 distribution, by considering the convolution of the orientation
 distribution with itself, or see the above remark.

  We also note that (\ref{eq:2.2}) is not a unique way to ensure
  dissipative system, and other relations for the rate of change of the
  lattice orientations which enforce dissipation may be possible. In
  this work, the particular assumption on the rate of change of the
  lattice orientation (\ref{eq:2.2}) is motivated by the approximation
  to the gradient flow dynamics near equilibrium
  \cite{DK:gbphysrev,MR3729587}. Experimental study of the dynamics of
  the lattice orientations/misorientations will be part of future
  research.



\end{remark}

Finally, as a part of $\frac{d}{dt}E(t) \le 0$ condition in
(\ref{eq:2a}), we also assume the dynamic boundary conditions for the
triple junctions, namely, for a constant $\eta>0$,
\begin{equation}
  \label{eq:2.3}
 \frac{d\vec{a}}{dt}(t)
  =
  \eta\sum_{j=1}^3\vec{T}^{(j)}(0,t),
  \quad t>0.
\end{equation}
This assumption implies that the last term in \eqref{eq:2a} satisfies,
\begin{equation}
 \label{eq:2.7}
  -\sum_{j=1}^3
  \vec{T}^{(j)}(0,t)
  \cdot
  \frac{d\vec{a}}{dt}(t)
  =
  -\frac{1}{\eta}
  \left|
    \frac{d\vec{a}}{dt}(t)
   \right|^2
  \leq0.
\end{equation}

Therefore,  we obtain from \eqref{eq:2.8}, \eqref{eq:2.6}, and
\eqref{eq:2.7}, that the entire system of grain boundaries $\Gamma_t^{(j)}$ is
dissipative, namely,
\begin{equation}\label{eq:2.8DE}
  \frac{d}{dt}E(t)
  =
  -\sum_{j=1}^3
  \int_{\Gamma_t^{(j)}}
   \frac{1}{\mu^{(j)}}|v_n^{(j)}|^2\,d\mathscr{H}^1
  -\frac1\eta
  \left|\frac{d\vec{a}}{dt}(t)\right|^2
  -\frac1\gamma
  \sum_{j=1}^3
  \left|\frac{d\alpha^{(j)}}{dt}\right|^2
  \leq 0.
\end{equation}
Finally, we combine
assumptions \eqref{eq:2.1}, \eqref{eq:2.2}, and \eqref{eq:2.3} to obtain
the following system of geometric evolution differential
equations to describe
motion of grain
boundaries $\Gamma^{(j)}_t, j=1, 2, 3$ together with a motion of the triple junction $\vec{a}(t)$:
\begin{equation}
 \label{eq:2.9}
 \left\{
  \begin{aligned}
   v_n^{(j)}
   &=
   \mu^{(j)}
   (\sigma^{(j)}_{\NorAngle\NorAngle}+\sigma^{(j)})
   \kappa^{(j)},\quad\text{on}\ \Gamma_t^{(j)},\ t>0,\quad j=1,2,3, \\
   \frac{d\alpha^{(j)}}{dt}
   &=
   -\gamma
   \biggl(
   \int_0^{1}
   \Bigl(
   \sigma^{(j+1)}_\MisOriAngle
   |\vec{b}^{(j+1)}|
   -
   \sigma^{(j)}_\MisOriAngle
   |\vec{b}^{(j)}|
   \Bigr)
   \,ds
   \biggr),\quad
   j=1,2,3,
   \\
   \frac{d\vec{a}}{dt}(t)
   &=
   \eta\sum_{k=1}^3\vec{T}^{(k)}(0,t)
   =
    \eta\sum_{k=1}^3
    (\sigma_\NorAngle^{(k)}\hat{\vec{n}}^{(k)}
    +\sigma^{(k)}\hat{\vec{b}}^{(k)}
    )(0,t),
   \quad t>0, \\
   \Gamma_t^{(j)}
   &:
   \vec{\xi}^{(j)}(s,t),\quad
   0\leq s\leq 1,\quad
   t>0,\quad
   j=1,2,3, \\
   \vec{a}(t)
   &=
   \vec{\xi}^{(1)}(0,t)
   =
   \vec{\xi}^{(2)}(0,t)
   =
   \vec{\xi}^{(3)}(0,t),
   \quad
   \text{and}
   \quad
   \vec{\xi}^{(j)}(1,t)=\vec{x}^{(j)},\quad
   j=1,2,3.
  \end{aligned}
  \right.
\end{equation}
\begin{remark}
The entire system \eqref{eq:2.9} satisfies energy
dissipation principle \eqref{eq:2.8DE}. However, it is important to
note, that there are three independent relaxation time scales in the
system \eqref{eq:2.9}, namely, $\mu^{(j)},
\gamma$ and $\eta$ (length, misorientation and position of the
triple junction). Classical approach is to let $\gamma
\rightarrow\infty$ and  $\eta \rightarrow\infty$.
\end{remark}
In this work, we let $\mu^{(j)}\rightarrow\infty$, and set $\gamma=\eta=1$ to study the effect of the dynamics of lattice orientations
$\alpha^{(j)}(t), j=1, 2, 3$ together with the effect of the dynamics of a triple
junction $\vec{a}(t)$ on a grain boundary motion.
Then, in this limit, $\Gamma^{(j)}_t$ becomes a line segment from the triple junction $\vec{a}(t)$
to the boundary point $\vec{x}^{(j)}$.  Hence, we have
\begin{equation*}
 \left\{
  \begin{aligned}
   \vec{\xi}^{(j)}(s,t)
   &=
   \vec{a}(t)+s\vec{b}^{(j)}(t),\quad
   0\leq s\leq 1,\quad
   t>0,\quad
   j=1,2,3, \\
   \vec{a}(t)+\vec{b}^{(j)}(t)
   &=
   \vec{x}^{(j)},
   \quad
   j=1,2,3.
  \end{aligned}
  \right.
\end{equation*}
 We assume that the surface tension $\sigma$ is independent
of the normal vector $\vec{n}$. Hereafter, we further assume the
following three conditions for the surface tension $\sigma$. First, we
assume positivity, namely, there exists a positive constant
$\Cl{const:2.Assumption1}>0$ such that,
\begin{equation}
 \label{eq:2.Assumption1}
  \sigma(\MisOriAngle)\geq \Cr{const:2.Assumption1},
\end{equation}
for $\MisOriAngle\in\R$. Second, we assume convexity, for all
$\MisOriAngle\in\R,$
\begin{equation}
 \label{eq:2.Assumption2}
  \sigma_\MisOriAngle(\MisOriAngle)\MisOriAngle\geq0.
\end{equation}
Furthermore, we assume,
\begin{equation}
 \label{eq:2.Assumption3}
  \sigma_\MisOriAngle(\MisOriAngle)=0
  \ \text{if and only if}\
  \MisOriAngle=0.
\end{equation}

\begin{remark}\label{rem:2.4}

1.  In this work we assume a more general surface energy
$\sigma(\Delta \alpha^{(j)})$ (\ref{eq:2.Assumption1}), (\ref{eq:2.Assumption3}), since
we consider a non-equilibrium state at time scale  $\mu^{(j)}\rightarrow\infty$ and
$\gamma=\eta=1$. Note that a different example of Read-Shockley
type surface energy \cite{ReadShockley} is the classical example of
the grain boundary energy derived under the assumption of
small misorientation angle $\Delta\alpha^{(j)}$, and the assumption of the
equilibrium state for a single fixed grain boundary at time scale
$\mu^{(j)}\rightarrow\infty, \eta \rightarrow\infty$ and $\gamma=0$.\\
2. In this work, for simplicity we consider cases of surface tensions without
normal dependence. This assumption is not as restrictive since our model is in terms of the orientation, instead of misorientation, as we had discussed in Remark \ref{rem:2.2}. Note also, that the convexity
  condition (\ref{eq:2.Assumption2}) is not
  needed for local existence results and dissipation estimates for the
  energy, Sections
  \ref{sec:4} -\ref{sec:5} and Section \ref{sec:7}. The condition
  (\ref{eq:2.Assumption2}) is essentially used to show the
  misorientation/orientation estimates, see Sections \ref{sec:3}, \ref{sec:5}
  and \ref{sec:7}, and, as a part of future work, we will investigate
  possibility of relaxing this assumption to derive similar estimates. In addition, in this work, to show uniqueness of
  the solution to (\ref{eq:4.2}), we proceed using
  misorientation/orientation estimates from Section \ref{sec:5}.
  However, one can obtain uniqueness result without the use of those
  estimates, and instead using the estimate \eqref{eq:4.23}, in proof of Theorem \ref{thm:4.1}.

Thus, the system of geometric evolution differential equations
\eqref{eq:2.9} becomes the following system of ordinary differential
equations (ODE):
\begin{equation}
 \label{eq:2.5}
 \left\{
  \begin{aligned}
   \frac{d\alpha^{(j)}}{dt}
   &=
   -
   \Bigl(
   \sigma_\MisOriAngle(\Delta\alpha^{(j+1)})|\vec{b}^{(j+1)}|
   -
   \sigma_\MisOriAngle(\Delta\alpha^{(j)})|\vec{b}^{(j)}|
   \Bigr)
   ,
   \quad
   j=1,2,3, \\
   \frac{d\vec{a}}{dt}(t)
   &=
   \sum_{j=1}^3
   \sigma(\Delta\alpha^{(j)})
   \frac{\vec{b}^{(j)}}{|\vec{b}^{(j)}|},
   \quad t>0, \\
   \vec{a}(t)+\vec{b}^{(j)}(t)
   &=
   \vec{x}^{(j)},
   \quad
   j=1,2,3.
  \end{aligned}
 \right.
\end{equation}
 Below, we continue with a study of the local well-posedness of
 the problem \eqref{eq:2.5} with the initial data given by
 $\alpha^{(1)}_0,\alpha^{(2)}_0,\alpha^{(3)}_0,\vec{a}_0$.
  \begin{remark}\label{rem:2.5}
  1. Note, that the reduced model (\ref{eq:2.5}) is not a
  standard ODE system. This is the ODE system where each variable is
  locally constrained. Moreover, local well-posedness result  (e.g. local
  existence result) for the
  original model (\ref{eq:2.9}) will not imply local well-posedness
  result for the reduced system  (\ref{eq:2.5}) (it is unknown
  if the reduced model  (\ref{eq:2.5}) is actually a small perturbation of
  (\ref{eq:2.9}).).\\
   2. The reduced model (\ref{eq:2.5}) captures the dynamics of the
 orientations/misorientations and the triple junctions, and at the same
 time is more accessible for the analysis than the model
 (\ref{eq:2.9}). In addition, the system (\ref{eq:2.5}) is
 consistent/motivated by the model in \cite{DK:gbphysrev,MR2772123}. The
 well-posedness analysis of (\ref{eq:2.5}) is a step towards similar
 analysis for the model in \cite{DK:gbphysrev,MR2772123}, as well as for
 the original system (\ref{eq:2.9}).
\end{remark}

\section{Equilibrium}
\label{sec:3}
We study an associated equilibrium solution of the system
\eqref{eq:2.5}, namely,
\begin{equation}
 \label{eq:3.3}
  \left\{
   \begin{aligned}
    0
    &=
    \Bigl(
    \sigma_\MisOriAngle(\Delta\alpha_\infty^{(j+1)})|\vec{b}_\infty^{(j+1)}|
    -
    \sigma_\MisOriAngle(\Delta\alpha_\infty^{(j)})|\vec{b}_\infty^{(j)}|
    \Bigr), \\
    \vec{0}
    &=
    \sum_{j=1}^3
    \left(
    \sigma(\Delta\alpha_\infty^{(j)})
    \right)
    \frac{\vec{b}_\infty^{(j)}}{|\vec{b}_\infty^{(j)}|},
    \\
    \vec{a}_\infty+\vec{b}^{(j)}_\infty
    &=
    \vec{x}^{(j)},\quad
    j=1,2,3.
   \end{aligned}
  \right.
\end{equation}
To consider the equilibrium system \eqref{eq:3.3}, we assume that each
Dirichlet point $\vec{x}^{(j)}$ does not coincide with the other
Dirichlet point.

\begin{lemma}
 \label{lem:3.2}
 Let
 $(\alpha_\infty^{(1)},\alpha_\infty^{(2)},\alpha_\infty^{(3)},\vec{a}_\infty)$
 be a solution of the equilibrium system \eqref{eq:3.3}. Assume
 \eqref{eq:2.Assumption2} and \eqref{eq:2.Assumption3}. Then
 $\alpha_\infty^{(1)}=\alpha_\infty^{(2)}=\alpha_\infty^{(3)}$.
\end{lemma}

\begin{proof}
 We multiply the first equation of \eqref{eq:3.3} by
 $\alpha_\infty^{(j)}$ and sum to $j=1,2,3$, to obtain
 \begin{equation}
   0= \sum_{j=1}^3
   \Bigl(
   \sigma_\MisOriAngle(\Delta\alpha_\infty^{(j+1)})|\vec{b}_\infty^{(j+1)}|
   -
   \sigma_\MisOriAngle(\Delta\alpha_\infty^{(j)})|\vec{b}_\infty^{(j)}|
   \Bigr)
   \alpha^{(j)}_\infty
   =
  \sum_{j=1}^3
  \Bigl(
  \sigma_\MisOriAngle(\Delta\alpha_\infty^{(j)})  |\vec{b}_\infty^{(j)}|
   \Bigr) \Delta\alpha^{(j)}_\infty.
 \end{equation}
 Note that, at least two of the terms $|\vec{b}_\infty^{(j)}|, j=1,2,3$ are non zero,
 otherwise it will contradict the assumption that the Dirichlet points
 $\vec{x}^{(j)}$ are distinct. Hence, from \eqref{eq:2.Assumption2}-\eqref{eq:2.Assumption3},
 we obtain that $\alpha_\infty^{(1)}=\alpha_\infty^{(2)}=\alpha_\infty^{(3)}$.
\end{proof}

%
%
%

From Lemma \ref{lem:3.2}, it follows that,
in the equilibrium state,
there is no lattice misorientation between neighboring grains that have
grain boundaries meeting at that triple junction.  As a consequence, the
equilibrium system \eqref{eq:3.3} becomes,
\begin{equation}
 \label{eq:3.6}
  \left\{
   \begin{aligned}
    \vec{0}
    &=
    \sum_{j=1}^3
    \frac{\vec{b}_\infty^{(j)}}{|\vec{b}_\infty^{(j)}|},
    \\
    \vec{a}_\infty+\vec{b}^{(j)}_\infty
    &=
    \vec{x}^{(j)},\quad
    j=1,2,3.
   \end{aligned}
  \right.
\end{equation}
The equation \eqref{eq:3.6} is related to the Fermat-Torricelli
problem. More precisely, if we have that,  for each $i=1, 2, 3$,
\begin{equation}
 \label{eq:3.5}
  \left|
   \sum_{j=1,\ i\neq j}^3
   \frac{\vec{x}^{(j)}-\vec{x}^{(i)}}{|\vec{x}^{(j)}-\vec{x}^{(i)}|}
  \right|
  >1,
\end{equation}
then $\vec{a}_\infty$ is the unique minimizer of the function,
 \begin{equation}
  \label{eq:3.7}
   f(\vec{a})
   =
   \sum_{j=1}^{3}|\vec{a}-\vec{x}^{(j)}|,
   \quad
   \vec{a}\in\R^2,
 \end{equation}
 and $\vec{a}_\infty\neq\vec{x}^{(j)}$ for $j=1,2,3$ (See \cite[Theorem
 18.28]{MR1677397}).  Note,  that the assumption \eqref{eq:3.5}
 satisfies if and only if all three angles of the triangle, formed by
 vertices located at the nodes $\vec{x}^{(1)}$, $\vec{x}^{(2)}$, $\vec{x}^{(3)}$, are less than $120^{\circ}$.

%

\section{Local existence}
\label{sec:4}
Here, we discuss local existence which validates the consistency of
the proposed model. Let $\vec{x}^{(j)}\in\R^2$, $\vec{\alpha}_0=(\alpha_0^{(1)},\alpha_0^{(2)},\alpha_0^{(3)})\in\R^3$, and
$\vec{a}_0\in\R^2$ be given initial data and we consider the local existence
of the problem of \eqref{eq:2.5}, namely
\begin{equation}
 \label{eq:4.2}
 \left\{
  \begin{aligned}
   \frac{d\alpha^{(j)}}{dt}
   &=
   -
   \Bigl(
   \sigma_\MisOriAngle(\Delta\alpha^{(j+1)})|\vec{b}^{(j+1)}|
   -
   \sigma_\MisOriAngle(\Delta\alpha^{(j)})|\vec{b}^{(j)}|
   \Bigr)
   ,
   \quad
   j=1,2,3, \\
   \frac{d\vec{a}}{dt}(t)
   &=
   \sum_{j=1}^3
   \sigma(\Delta\alpha^{(j)})
   \frac{\vec{b}^{(j)}}{|\vec{b}^{(j)}|},
   \quad t>0, \\
   \vec{\alpha}(t)
   &=
   (\alpha^{(1)}(t),\alpha^{(2)}(t),\alpha^{(3)}(t)),\quad
   t>0,\\
   \vec{a}(t)+\vec{b}^{(j)}(t)
   &=
   \vec{x}^{(j)},\quad
   t>0,
   \quad
   j=1,2,3, \\
   \vec{\alpha}(0)&=\vec{\alpha}_0,\quad
   \vec{a}(0)=\vec{a}_0.
  \end{aligned}
  \right.
\end{equation}
Assume for each $i=1,2,3$,
\begin{equation}
 \label{eq:4.13}
  \left|
   \sum_{j=1,\ i\neq j.}^3
   \frac{\vec{x}^{(j)}-\vec{x}^{(i)}}{|\vec{x}^{(j)}-\vec{x}^{(i)}|}
  \right|
  >1.
\end{equation}
We denote by $\vec{a}_\infty\neq \vec{x}^{(j)}$ for each
$j=1,2,3$, a solution to the system,
\begin{equation}
 \label{eq:4.3}
  \left\{
   \begin{aligned}
    \vec{0}
    &=
    \sum_{j=1}^3
    \frac{\vec{b}_\infty^{(j)}}{|\vec{b}_\infty^{(j)}|},
    \\
    \vec{a}_\infty+\vec{b}^{(j)}_\infty
    &=
    \vec{x}^{(j)},
    \quad
    j=1,2,3.
   \end{aligned}
  \right.
\end{equation}
The point $\vec{a}_\infty$  is a triple junction point (see Section \ref{sec:3}).

\begin{theorem}[Local existence]
 \label{thm:4.1}
 Let $\vec{x}^{(1)}$, $\vec{x}^{(2)}$, $\vec{x}^{(3)}\in \R^2$,
 $\vec{a}_0\in\R^2$, and $\vec{\alpha}_0\in\R^3$ be given initial
 data.  Assume condition
 \eqref{eq:4.13} for $i=1,2,3$,  and let $\vec{a}_\infty$ be a solution of
 \eqref{eq:4.3}. Further,  assume that for all $j=1,2,3$,
 \begin{equation}
  \label{eq:4.10}
   |\vec{a}_0-\vec{a}_\infty|<\frac12|\vec{b}^{(j)}_\infty|.
 \end{equation}
 Then, there exists a local in time solution $(\vec{\alpha}, \vec{a})$
 of \eqref{eq:4.2} on $[0,T_{\text{max}})$, such that
 \begin{equation}
  \label{eq:4.21}
  |\vec{a}(t)-\vec{a}_\infty|<|\vec{b}^{(j)}_\infty|
   \quad
   \text{for all}\quad
   j=1,2,3,\ \text{and}\ 
   0\leq t<T_{\text{max}}.
 \end{equation}
 Furthermore, the maximal existence time $T_{\text{max}}$ of the solution is estimated by 
 \begin{equation}
  \label{eq:4.22}
  T_{\text{max}}
   \geq
   \min
  \left\{
   \frac{|\vec{\alpha}_0|}{4(M_1+8M_2|\vec{\alpha}_0|)\sum_{j=1}^3|\vec{b}^{(j)}_\infty|},\
   \frac{|\vec{a}_0-\vec{a}_\infty|}{3M_0}
  ,\
  \frac{1}{12M_1},\
  \frac{1}{8M_0\sum_{j=1}^3\frac{1}{|\vec{b}_\infty^{(j)}|-2|\vec{a}_0-\vec{a}_\infty|}}
 \right\},
 \end{equation}
 where 
\begin{equation*}
 M_0:=\sup_{|\MisOriAngle|
  \leq
  4|\vec{\alpha}_0|}|\sigma(\MisOriAngle)|,\quad
  M_1:=\sup_{|\MisOriAngle|\leq
  4|\vec{\alpha}_0|}|\sigma_\MisOriAngle(\MisOriAngle)|,\quad
  M_2
  :=
  \sup_{|\MisOriAngle_1|, |\MisOriAngle_2|\leq 4|\vec{\alpha}_0|}
  \frac{|\sigma_\MisOriAngle(\MisOriAngle_1)-\sigma_\MisOriAngle(\MisOriAngle_2)|}
  {|\MisOriAngle_1-\MisOriAngle_2|}.
\end{equation*}
\end{theorem}

\begin{remark}
 The Theorem \ref{thm:4.1} provides not only existence of
 the local in time
 solution for the model \eqref{eq:4.2}, but it also
 gives the local existence of the triple
 junction. The estimate \eqref{eq:4.21}
 guarantees that $\vec{a}(t)$ is the position of the triple junction formed by the
 grain boundaries $\vec{x}^{(j)}-\vec{a}(t)$. Note,  if $\vec{a}(t)$
 is sufficiently far
 from the position of the triple junction $\vec{a}_\infty$ of the equilibrium state, for
 instance if $\vec{x}^{(j)}-\vec{x}^{(k)}$ is a part of
 $\vec{a}(t)-\vec{x}^{(k)}$, then $\vec{a}(t)$ might not be the triple
 junction. Further, (\ref{eq:4.22}) gives the explicit dependence of the maximal
 existence time $T_{\text{max}}$ on $|\vec{a}_0-\vec{a}_\infty|$. This
 is an important result for the analysis of the global in time solution which
 will be part of a forthcoming work.
\end{remark}

To show Theorem \ref{thm:4.1}, we construct a contraction mapping on a complete metric space. Let $\Cl{const:4.1}$, $\Cl{const:4.2}>0$ and
$T>0$ be positive constants that we will define later, and denote,
\begin{equation*}
 X_T:=
 \{
 (\vec{\alpha},\vec{a})\in C([0,T]\,;\,\R^3\times\R^2),\
 \|\vec{\alpha}\|_{C([0,T])}\leq \Cr{const:4.1},\
 \|\vec{a}-\vec{a}_\infty\|_{C([0,T])}\leq \Cr{const:4.2}
 \}.
\end{equation*}

Note that in the definition of the space $X_T$,  we
  use the position of the triple junction $\vec{a}_\infty$ at the equilibrium
state as the point of reference, rather than the
position of the
triple junction $\vec{a}_0$ at the initial time as one would consider in
the classical ODE theory. Such definition of the space $X_T$ is
employed in order to obtain the estimates on the position of the triple junctions from the one of the equilibrium state $\vec{a}_\infty$,
\eqref{eq:4.21}, as well as to derive the maximal existence time estimate
\eqref{eq:4.22}.

Next, define for $(\vec{\alpha},\vec{a})\in X_T$, $i=1,2,3$, and $t>0$
\begin{equation}
 \begin{split}
  \Phi^{(i)}(\vec{\alpha},\vec{a})(t)
  &:=\alpha^{(i)}_0
  -\int_0^t
  \left(
  \sigma_\MisOriAngle(\Delta\alpha^{(j+1)}(\tau))|\vec{b}^{(j+1)}(\tau)|
  -
  \sigma_\MisOriAngle(\Delta\alpha^{(j)}(\tau))|\vec{b}^{(j)}(\tau)|
  \right)
  \,d\tau, \\
  \Psi(\vec{\alpha},\vec{a})(t)
  &:=\vec{a}_0
  +\sum_{j=1}^3
  \int_0^t
  \sigma(\alpha(\tau))
  \frac{\vec{b}^{(j)}(\tau)}{|\vec{b}^{(j)}(\tau)|}
  \,d\tau,
 \end{split}
\end{equation}
where $\vec{b}^{(j)}(\tau)=\vec{x}^{(j)}-\vec{a}(\tau)$. Our goal now is
to show that $(\Phi=(\Phi^{(1)},\Phi^{(2)},\Phi^{(3)}),\Psi)$ is a
contraction mapping on $X_T$ for the appropriate choice of positive
constants $\Cr{const:4.1}$, $\Cr{const:4.2}$, and $T>0$.
Hereafter we define,
\begin{equation*}
 M_0:=\sup_{|\MisOriAngle|
  \leq
  2C_2}|\sigma(\MisOriAngle)|,\quad
  M_1:=\sup_{|\MisOriAngle|\leq
  2C_2}|\sigma_\MisOriAngle(\MisOriAngle)|,\quad
  M_2
  :=
  \sup_{|\MisOriAngle_1|, |\MisOriAngle_2|\leq 2\Cr{const:4.1}}
  \frac{|\sigma_\MisOriAngle(\MisOriAngle_1)-\sigma_\MisOriAngle(\MisOriAngle_2)|}
  {|\MisOriAngle_1-\MisOriAngle_2|}.
\end{equation*}
Later, the constant $\Cr{const:4.1}$ will be taken to be
$2|\vec{\alpha}_0|$. Next, two Lemmas \ref{lem:4.2}, \ref{lem:4.3} show that $\Phi$ and $\Psi$ is a map on $X_T$.
 \begin{lemma}
 \label{lem:4.2}
 If the conditions below are satisfied,
 \begin{equation}
  \label{eq:4.4}
   2|\vec{\alpha}_0|\leq \Cr{const:4.1},
 \end{equation}
 and
 \begin{equation}
  \label{eq:4.5}
   (2M_1+4M_2\Cr{const:4.1})
   (|\vec{b}^{(1)}_\infty|
   +|\vec{b}^{(2)}_\infty|+|\vec{b}^{(3)}_\infty|+3\Cr{const:4.2})T
   \leq
   \frac12\Cr{const:4.1},
 \end{equation}
  then $|\Phi(\vec{\alpha},\vec{a})|\leq\Cr{const:4.1}$ for all
 $(\vec{\alpha},\vec{a})\in X_T$.
 \end{lemma}
\begin{proof}
 [Proof of Lemma \ref{lem:4.2}]
 By the triangle inequality, for $0\leq t\leq T$,
 \[
 \begin{split}
  |\Phi(\vec{\alpha},\vec{a})(t)|
  &\leq
  |\vec{\alpha}_0|
  +
  \sum_{j=1}^3
  \left|
  \int_0^t
  \left(
  \sigma_\MisOriAngle(\Delta\alpha^{(j+1)}(\tau))|\vec{b}^{(j+1)}(\tau)|
  -
  \sigma_\MisOriAngle(\Delta\alpha^{(j)}(\tau))|\vec{b}^{(j)}(\tau)|
  \right)
  \,d\tau
  \right| \\
  &\leq
  |\vec{\alpha}_0|
  +
  \sum_{j=1}^3
  \bigg(
  \int_0^t
  |
  \sigma_\MisOriAngle(\Delta\alpha^{(j+1)}(\tau))
  -
  \sigma_\MisOriAngle(\Delta\alpha^{(j)}(\tau))
  ||\vec{b}^{(j+1)}(\tau)|
  \,d\tau \\
  &\quad
  +
  \int_0^t
  |\sigma_\MisOriAngle(\Delta\alpha^{(j)}(\tau))|
  \left|
  |\vec{b}^{(j+1)}(\tau)|
  -
  |\vec{b}^{(j)}(\tau)|
  \right|
  \,d\tau
  \bigg).
\end{split}
 \]
 Next, using that $|\Delta\alpha^{(j)}|\leq 2\Cr{const:4.1}$, and that,
 \begin{equation*}
  |
   \sigma_\MisOriAngle(\Delta\alpha^{(j+1)}(\tau))
   -
   \sigma_\MisOriAngle(\Delta\alpha^{(j)}(\tau))
   |
   \leq
   M_2|\Delta\alpha^{(j+1)}(\tau)-\Delta\alpha^{(j)}(\tau)|
   \leq
   4M_2\Cr{const:4.1},
 \end{equation*} 
 we have that,
 \begin{equation*}
  |\Phi(\vec{\alpha},\vec{a})(t)| \leq
   |\vec{\alpha}_0|
   +
   (2M_1+4M_2\Cr{const:4.1})T
   \sum_{j=1}^3
   \sup_{0\leq \tau\leq T}|\vec{b}^{(j)}(\tau)|.
 \end{equation*}

 On the other hand, for $j=1,2,3$,
 \begin{equation}
  \label{eq:4.16}
   |\vec{b}^{(j)}(t)|
   =|\vec{x}^{(j)}-\vec{a}_\infty+\vec{a}_\infty-\vec{a}(t)|
   \leq |\vec{b}^{(j)}_\infty|+\Cr{const:4.2}.
 \end{equation}
 Therefore, from \eqref{eq:4.4} and \eqref{eq:4.5},
 \[
 |\Phi(\vec{\alpha},\vec{a})(t)|
 \leq
 |\vec{\alpha}_0|
 +
 (2M_1+4M_2\Cr{const:4.1})
 (|\vec{b}^{(1)}_\infty|
 +|\vec{b}^{(2)}_\infty|+|\vec{b}^{(3)}_\infty|+3\Cr{const:4.2})T
 \leq
 \Cr{const:4.1}.
 \]
\end{proof}

\begin{lemma}
 \label{lem:4.3}
 Assume for $j=1,2,3$ we have that,
 \begin{equation}
  \label{eq:4.6}
   \Cr{const:4.2}
   < |\vec{b}_\infty^{(j)}|.
 \end{equation}
 Then, $0<|\vec{b}_\infty^{(j)}|-\Cr{const:4.2} \leq
 |\vec{b}^{(j)}(t)|\leq |\vec{b}^{(j)}_\infty|+\Cr{const:4.2},$
 for all $j=1,2,3$, $(\vec{\alpha},\vec{a})\in X_T$, and $0\leq t\leq
 T$. Further if
 \begin{equation}
  \label{eq:4.9}
   2|\vec{a}_0-\vec{a}_\infty|\leq \Cr{const:4.2},
 \end{equation}
 and
 \begin{equation}
  \label{eq:4.7}
   3M_0T\leq \frac12\Cr{const:4.2},
 \end{equation}
 then $|\Psi(\vec{\alpha},\vec{a})(t)-\vec{a}_\infty|\leq\Cr{const:4.2},$
  for all $(\vec{\alpha},\vec{a})\in X_T$ and $0\leq t\leq T$.
\end{lemma}
\begin{proof}
 [Proof of Lemma \ref{lem:4.3}]
 For $(\vec{\alpha},\vec{a})\in X_T$, and $0\leq t\leq T$
 \[
 |\vec{b}_\infty^{(j)}|
 =
 |\vec{x}^{(j)}-\vec{a}(t)+\vec{a}(t)-\vec{a}_\infty|
 \leq
 |\vec{b}^{(j)}(t)|+|\vec{a}(t)-\vec{a}_\infty|
 \leq
 |\vec{b}^{(j)}(t)|+\Cr{const:4.2},
 \]
 thus we obtain $0<|\vec{b}_\infty^{(j)}|-\Cr{const:4.2} \leq
 |\vec{b}^{(j)}(t)|$. And  $|\vec{b}^{(j)}(t)|\leq
 |\vec{b}^{(j)}_\infty|+\Cr{const:4.2}$ follows from \eqref{eq:4.16}.
 To show estimate
 $|\Psi(\vec{\alpha},\vec{a})(t)-\vec{a}_\infty|\leq\Cr{const:4.2}$,
 we use the assumption \eqref{eq:4.9} and \eqref{eq:4.7}, to obtain that
 for any $(\vec{\alpha},\vec{a})\in X_T$,
 \[
 \begin{split}
  |\Psi(\vec{\alpha},\vec{a})(t)-\vec{a}_\infty|
  &\leq
  |\vec{a}_0-\vec{a}_\infty|
  +
  \sum_{j=1}^3
  \left|
  \int_0^t
  \sigma(\Delta\alpha^{(j)}(\tau))
  \frac{\vec{b}^{(j)}(\tau)}{|\vec{b}^{(j)}(\tau)|}
  \,d\tau
  \right| \\
  &\leq
  \frac12\Cr{const:4.2}
  +\sum_{j=1}^3
  \sup_{0\leq \tau\leq T}\sigma(\Delta\alpha^{(j)}(\tau))
  T \\
  &\leq
  \frac12\Cr{const:4.2}
  +
  3
  M_0
  T
  \leq
  \Cr{const:4.2},
 \end{split}
 \]
 for all $0\leq t\leq T$.
\end{proof}

The next two Lemmas \ref{lem:4.4} and \ref{lem:4.5} give the Lipschitz property of the map $(\Phi, \Psi)$.

\begin{lemma}[Lipschitz estimates]
 \label{lem:4.4}
 For $(\vec{\alpha}_1,\vec{a}_1)$, $(\vec{\alpha}_2,\vec{a}_2)\in
 X_T$, we have that
 \begin{equation}
  \label{eq:4.8}
  \begin{split}
   &\quad
   \|\Phi(\vec{\alpha}_1,\vec{a}_1)
   -\Phi(\vec{\alpha}_2,\vec{a}_2)\|_{C([0,T])} \\
   &\leq
   4M_2
   (
   |\vec{b}_\infty^{(1)}|
   +
   |\vec{b}_\infty^{(2)}|
   +
   |\vec{b}_\infty^{(3)}|
   +
   3\Cr{const:4.2})T
   \|
   \vec{\alpha}_1-\vec{\alpha}_2
   \|_{C([0,T])}
   +
   6M_1T
   \|\vec{a}_1-\vec{a}_2\|_{C([0,T])}
   .
  \end{split}
 \end{equation}
\end{lemma}

\begin{proof}
 [Proof of Lemma \ref{lem:4.4}]
 For $0\leq t\leq T$, by the Lipschitz continuity of
 $\sigma_\MisOriAngle$ we obtain that
 \[
 \begin{split}
  &\quad
  |\Phi(\vec{\alpha}_1,\vec{a}_1)(t)
  -\Phi(\vec{\alpha}_2,\vec{a}_2)(t)| \\
  &\leq
  \sum_{j=1}^3
  \left|
  \int_0^t
  \left(
  \sigma_\MisOriAngle(\Delta\alpha^{(j+1)}_1)|\vec{b}^{(j+1)}_1|
  -
  \sigma_\MisOriAngle(\Delta\alpha^{(j)}_1)|\vec{b}^{(j)}_1|
  -
  \sigma_\MisOriAngle(\Delta\alpha^{(j+1)}_2)|\vec{b}^{(j+1)}_2|
  +
  \sigma_\MisOriAngle(\Delta\alpha^{(j)}_2)|\vec{b}^{(j)}_2|
  \right)
  \,d\tau
  \right| \\
  &\leq
  \sum_{j=1}^3
  \int_0^t
  \left(
  \left|
  \sigma_\MisOriAngle(\Delta\alpha^{(j+1)}_1)
  \right|
  \left|
  |\vec{b}^{(j+1)}_1|
  -
  |\vec{b}^{(j+1)}_2|
  \right|
  +
  \left|
  \sigma_\MisOriAngle(\Delta\alpha^{(j+1)}_1)
  -
  \sigma_\MisOriAngle(\Delta\alpha^{(j+1)}_2)
  \right|
  |\vec{b}^{(j+1)}_2|
  \right)
  \,d\tau  \\
  &\quad
  +
  \sum_{j=1}^3
  \int_0^t
  \left(
  \left|
  \sigma_\MisOriAngle(\Delta\alpha^{(j)}_1)
  \right|
  \left|
  |\vec{b}^{(j)}_1|
  -
  |\vec{b}^{(j)}_2|
  \right|
  +
  \left|
  \sigma_\MisOriAngle(\Delta\alpha^{(j)}_1)
  -
  \sigma_\MisOriAngle(\Delta\alpha^{(j)}_2)
  \right|
  |\vec{b}^{(j)}_2|
  \right)
  \,d\tau.
\end{split}
 \]
 Next, using  $\vec{b}^{(j)}_k=\vec{x}^{(j)}-\vec{a}_k,$
$\Delta\alpha^{(j)}=\alpha^{(j-1)}-\alpha^{(j)}$ and (\ref{eq:4.16}),
we have,
\[  |\Phi(\vec{\alpha}_1,\vec{a}_1)(t)
  -\Phi(\vec{\alpha}_2,\vec{a}_2)(t)|\leq
  6M_1T
  \|\vec{a}_1-\vec{a}_2\|_{C([0,T])}
  +
  4M_2
  (
  |\vec{b}_\infty^{(1)}|
  +
  |\vec{b}_\infty^{(2)}|
  +
  |\vec{b}_\infty^{(3)}|
  +
  3\Cr{const:4.2})T
  \|
  \vec{\alpha}_1-\vec{\alpha}_2
  \|_{C([0,T])}.
 \]
 Thus, we obtain the inequality \eqref{eq:4.8}.

\end{proof}

\begin{lemma}[Lipschitz estimates]
 \label{lem:4.5}
 Assume condition \eqref{eq:4.6} holds true. Then for
 $(\vec{\alpha}_1,\vec{a}_1)$, $(\vec{\alpha}_2,\vec{a}_2)\in X_T$, we
 have that
 \begin{equation}
  \label{eq:4.11}
  \begin{split}
   &\quad
   \|\Psi(\vec{\alpha}_1,\vec{a}_1)(t)
   -\Psi(\vec{\alpha}_2,\vec{a}_2)(t)\|_{C([0,T])} \\
   &\leq
   6M_1
   T\|\vec{\alpha}_1-\vec{\alpha}_2\|_{C([0,T])}
   +
   2M_0
   \left(
   \frac{1}{|\vec{b}_\infty^{(1)}|-\Cr{const:4.2}}
   +
   \frac{1}{|\vec{b}_\infty^{(2)}|-\Cr{const:4.2}}
   +
   \frac{1}{|\vec{b}_\infty^{(3)}|-\Cr{const:4.2}}
   \right)T
   \|\vec{a}_1-\vec{a}_2\|_{C([0,T])}.
  \end{split}
 \end{equation}
\end{lemma}

\begin{proof}
 [Proof of Lemma \ref{lem:4.5}]
 For $k=1,2$, denote
 $\sigma^{(j)}_k(t):=
 \sigma(\Delta\alpha^{(j)}_k(t))
 $. For
 $0\leq t\leq T$, we can obtain the following estimate
 \[
 \begin{split}
  |\Psi(\vec{\alpha}_1,\vec{a}_1)(t)
  -\Psi(\vec{\alpha}_2,\vec{a}_2)(t)|
  &=
  \left|
  \sum_{j=1}^3
  \int_0^t
  \left(
  \sigma^{(j)}_1(\tau)\frac{\vec{b}_1^{(j)}(\tau)}{|\vec{b}_1^{(j)}(\tau)|}
  -
  \sigma^{(j)}_2(\tau)\frac{\vec{b}_2^{(j)}(\tau)}{|\vec{b}_2^{(j)}(\tau)|}
  \right)
  \,d\tau
  \right| \\
  &\leq
  \sum_{j=1}^3
  \int_0^T
  \left|
  \sigma^{(j)}_1(\tau)\frac{\vec{b}_1^{(j)}(\tau)}{|\vec{b}_1^{(j)}(\tau)|}
  -
  \sigma^{(j)}_2(\tau)\frac{\vec{b}_2^{(j)}(\tau)}{|\vec{b}_2^{(j)}(\tau)|}
  \right|
  \,d\tau \\
  &\leq
  \sum_{j=1}^3
  \int_0^T
  \left|
  \sigma^{(j)}_1(\tau)-\sigma^{(j)}_2(\tau)
  \right|
  \,d\tau \\
  &\quad
  +
  \sum_{j=1}^3
  \int_0^T
  \sigma^{(j)}_2(\tau)
  \left|
  \frac{\vec{b}_1^{(j)}(\tau)}{|\vec{b}_1^{(j)}(\tau)|}
  -
  \frac{\vec{b}_2^{(j)}(\tau)}{|\vec{b}_2^{(j)}(\tau)|}
  \right|
  \,d\tau.
 \end{split}
 \]
 Since $(\vec{\alpha}_k,\vec{a}_k)\in X_T$,  we have
 \[
 \begin{split}
  \left|
  \sigma^{(j)}_1(\tau)-\sigma^{(j)}_2(\tau)
  \right|
  &=
  |
  \sigma(\Delta\alpha_1^{(j)}(\tau))
  -
  \sigma(\Delta\alpha_2^{(j)}(\tau))
  | \\
  &\leq
  M_1
  |
  \Delta\alpha_1^{(j)}(\tau)
  -
  \Delta\alpha_2^{(j)}(\tau)
  | \\
  &\leq
  2 M_1
  \|
  \vec{\alpha}_1
  -
  \vec{\alpha}_2
  \|_{C([0,T])}.
 \end{split}
 \]
 Hence, we derive that
 \[
 \sum_{j=1}^3
 \int_0^T
  \left|
 \sigma^{(j)}_1(\tau)-\sigma^{(j)}_2(\tau)
 \right|
 \,d\tau
 \leq
 6M_1T
 \|\vec{\alpha}_1-\vec{\alpha}_2\|_{C([0,T])}.
 \]

 Next, due to condition \eqref{eq:4.6}, we can apply Lemma
 \ref{lem:4.3}. Therefore, we have that
 $|\vec{b}_k^{(j)}(\tau)|\neq0$ for $j=1,2,3$, $k=1,2$, and $0\leq
 \tau\leq T$. By direct
 calculations, we have that
 \begin{equation}
  \label{eq:4.12}
 \begin{split}
   \left|
   \frac{\vec{b}_1^{(j)}(\tau)}{|\vec{b}_1^{(j)}(\tau)|}
   -
   \frac{\vec{b}_2^{(j)}(\tau)}{|\vec{b}_2^{(j)}(\tau)|}
   \right|
   &=
   \frac{1}{|\vec{b}_1^{(j)}(\tau)|}
   \left|
   \vec{b}_1^{(j)}(\tau)-
   \frac{|\vec{b}_1^{(j)}(\tau)|}{|\vec{b}_2^{(j)}(\tau)|} \vec{b}_2^{(j)}(\tau)
   \right| \\
   &\leq
   \frac{1}{|\vec{b}_1^{(j)}(\tau)|}
   \left(
   \left|
   \vec{b}_1^{(j)}(\tau)-\vec{b}_2^{(j)}(\tau)
   \right|
   +
   \left|
   \left(
   1-\frac{|\vec{b}_1^{(j)}(\tau)|}{|\vec{b}_2^{(j)}(\tau)|}
   \right) \vec{b}_2^{(j)}(\tau)
   \right|
   \right) \\
   &\leq
   \frac{1}{|\vec{b}_1^{(j)}(\tau)|}
   \left(
   \left|
   \vec{b}_1^{(j)}(\tau)-\vec{b}_2^{(j)}(\tau)
   \right|
   +
   \left|
   |\vec{b}_2^{(j)}(\tau)|
   -
   |\vec{b}_1^{(j)}(\tau)|
   \right|
   \right) \\
   &\leq
   \frac{2}{|\vec{b}_1^{(j)}(\tau)|}
   \left|
   \vec{b}_1^{(j)}(\tau)-\vec{b}_2^{(j)}(\tau)
   \right|.
 \end{split}
 \end{equation}
 Again, using Lemma \ref{lem:4.3}, and due to uniqueness of the point $\vec{a}_{\infty}$ (see Section \ref{sec:3}), we have
 that $0<|\vec{b}_\infty^{(j)}|-\Cr{const:4.2} \leq
 |\vec{b}_1^{(j)}(\tau)|$ for $j=1,2,3$, and $0\leq \tau\leq T$. Thus,
 we derive that
 \[
 \left|
 \frac{\vec{b}_1^{(j)}(\tau)}{|\vec{b}_1^{(j)}(\tau)|}
 -
 \frac{\vec{b}_2^{(j)}(\tau)}{|\vec{b}_2^{(j)}(\tau)|}
 \right|
 \leq \frac{2}{|\vec{b}_\infty^{(j)}|-\Cr{const:4.2}}
 \|\vec{a}_1-\vec{a}_2\|_{C([0,T])},
 \]
 and,
\begin{equation*}
  \begin{split}
   &\quad\sum_{j=1}^3
   \int_0^T
   \sigma^{(j)}_2(\tau)
   \left|
   \frac{\vec{b}_1^{(j)}(\tau)}{|\vec{b}_1^{(j)}(\tau)|}
   -
   \frac{\vec{b}_2^{(j)}(\tau)}{|\vec{b}_2^{(j)}(\tau)|}
   \right|
   \,d\tau \\
   &\leq
   \sum_{j=1}^3
   \int_0^T
   \frac{2M_0}{|\vec{b}_\infty^{(j)}|-\Cr{const:4.2}}
   \|\vec{a}_1-\vec{a}_2\|_{C([0,T])}
   \,d\tau \\
   &\leq
   2M_0
   \left(
   \frac{1}{|\vec{b}_\infty^{(1)}|-\Cr{const:4.2}}
   +
   \frac{1}{|\vec{b}_\infty^{(2)}|-\Cr{const:4.2}}
   +
   \frac{1}{|\vec{b}_\infty^{(3)}|-\Cr{const:4.2}}
   \right)T
   \|\vec{a}_1-\vec{a}_2\|_{C([0,T])}.
  \end{split}
\end{equation*}
Hence, we obtain the desired estimate,
 \[
 \begin{split}
  &\quad
  |\Psi(\vec{\alpha}_1,\vec{a}_1)(t)
  -\Psi(\vec{\alpha}_2,\vec{a}_2)(t)| \\
  &\leq
  6M_1
  T\|\vec{\alpha}_1-\vec{\alpha}_2\|_{C([0,T])}
  +
  2M_0
  \left(
  \frac{1}{|\vec{b}_\infty^{(1)}|-\Cr{const:4.2}}
  +
  \frac{1}{|\vec{b}_\infty^{(2)}|-\Cr{const:4.2}}
  +
  \frac{1}{|\vec{b}_\infty^{(3)}|-\Cr{const:4.2}}
  \right)T
  \|\vec{a}_1-\vec{a}_2\|_{C([0,T])}.
 \end{split}
 \]
\end{proof}
\begin{proof}
 [Proof of Theorem \ref{thm:4.1}]
 We start with given constants $\Cr{const:4.1}$ and $\Cr{const:4.2}$ for, $\Cr{const:4.1}:=2|\vec{\alpha}_0|$ and
 $\Cr{const:4.2}:=2|\vec{a}_0-\vec{a}_\infty|$.  Note,  that due to
 assumption \eqref{eq:4.10}, we obtain that
 $\Cr{const:4.2}<|\vec{b}_\infty^{(j)}|$ for all $j=1,2,3$, and hence,
 we have that,
 \[
 |\vec{b}_\infty^{(1)}|+|\vec{b}_\infty^{(2)}|+|\vec{b}_\infty^{(3)}|
 +3\Cr{const:4.2}
 \leq
 2(|\vec{b}_\infty^{(1)}|+|\vec{b}_\infty^{(2)}|+|\vec{b}_\infty^{(3)}|).
 \]
Next, we will find the bound for the existence time
   $T$ which will guarantee the contraction mapping on $X_T$. Take
   time $T>0$ as defined below,
\begin{equation}\label{timeest}
 T:=
  \min
  \left\{
   \frac{\Cr{const:4.1}}{8(M_1+4M_2\Cr{const:4.1})\sum_{j=1}^3|\vec{b}^{(j)}_\infty|},\
   \frac{\Cr{const:4.2}}{6M_0}
  ,\
  \frac{1}{12M_1},\
  \frac{1}{8M_0\sum_{j=1}^3\frac{1}{|\vec{b}_\infty^{(j)}|-\Cr{const:4.2}}}
 \right\}.
\end{equation}
Recall, that the space $X_T$ (see Section \ref{sec:4}) is a complete metric space endowed with a distance
 \[
 d_{X_T}((\vec{\alpha}_1,\vec{a}_1),(\vec{\alpha}_2,\vec{a}_2))
 =\|\vec{\alpha}_1-\vec{\alpha}_2\|_{C([0,T])}
 +\|\vec{a}_1-\vec{a}_2\|_{C([0,T])}.
 \]
In addition, definition of constants $\Cr{const:4.1}$ and $\Cr{const:4.2}$ above
implies condition
 \eqref{eq:4.4}, \eqref{eq:4.6}, and \eqref{eq:4.9} in Lemmas
 \ref{lem:4.2}-\ref{lem:4.3}.
 Moreover, since we selected $T$, as
 \[
 T\leq
 \frac{\Cr{const:4.1}}{8(M_1+4M_2\Cr{const:4.1})\sum_{j=1}^3|\vec{b}^{(j)}_\infty|}
 \ \text{and}\
 T
 \leq
 \frac{\Cr{const:4.2}}{6M_0},
 \]
 we also have that,
 \[
 (2M_1+4M_2\Cr{const:4.1})(|\vec{b}^{(1)}_\infty|
 +|\vec{b}^{(2)}_\infty|+|\vec{b}^{(3)}_\infty|+3\Cr{const:4.2})T
 \leq
 4(M_1+4M_2\Cr{const:4.1})(|\vec{b}^{(1)}_\infty|
 +|\vec{b}^{(2)}_\infty|+|\vec{b}^{(3)}_\infty|)T
 \leq
 \frac12{\Cr{const:4.1}},
 \]
 and
 \[
 3M_0T\leq \frac12\Cr{const:4.2}
 \]
 Thus, the other conditions \eqref{eq:4.5} and \eqref{eq:4.7}  in Lemmas
 \ref{lem:4.2}-\ref{lem:4.3} are also satisfied. Therefore, we can employ
 Lemmas \ref{lem:4.2} and \ref{lem:4.3} to show that the mapping
 \[
 X_T\ni(\vec{\alpha},\vec{a})\mapsto
 (\Phi(\vec{\alpha},\vec{a}),\Psi(\vec{\alpha},\vec{a}))\in X_T
 \]
 is well-defined. Next,  combining estimates \eqref{eq:4.8} and
 \eqref{eq:4.11} in Lemmas \ref{lem:4.4}-\ref{lem:4.5} together, we obtain that,
 \begin{equation}
  \label{eq:4.13d}
 \begin{split}
  &\quad
  d_X
  ((\Phi(\vec{\alpha}_1,\vec{a}_1), \Psi(\vec{\alpha}_1,\vec{a}_1)),
  (\Phi(\vec{\alpha}_2,\vec{a}_2), \Psi(\vec{\alpha}_2,\vec{a}_2))
  ) \\
  &\leq
  \left(
  6M_1+ 8M_2(|\vec{b}^{(1)}_\infty|
  +|\vec{b}^{(2)}_\infty|+|\vec{b}^{(3)}_\infty|)
  \right)
  T
  \|\vec{\alpha}_1-\vec{\alpha}_2\|_{C([0,T])} \\
  &\quad
  +
  \left(
  6M_1
  +
  2M_0
  \left(
  \frac{1}{|\vec{b}_\infty^{(1)}|-\Cr{const:4.2}}
  +
  \frac{1}{|\vec{b}_\infty^{(2)}|-\Cr{const:4.2}}
  +
  \frac{1}{|\vec{b}_\infty^{(3)}|-\Cr{const:4.2}}
  \right)
  \right)
  T\|\vec{a}_1-\vec{a}_2\|_{C([0,T])} \\
 \end{split}
 \end{equation}
 for $(\vec{\alpha}_1,\vec{a}_1)$, $(\vec{\alpha}_2,\vec{a}_2)\in X_T$.
 Next, since we selected time $T$ as in (\ref{timeest})
 we have that,
 \begin{equation}
  \label{eq:4.14t}
   T\leq
   \frac{\Cr{const:4.1}}{8(M_1+4M_2\Cr{const:4.1})\sum_{j=1}^3|\vec{b}^{(j)}_\infty|}
  \leq
  \frac{1}{32M_2\sum_{j=1}^3|\vec{b}^{(j)}_\infty|}
  ,\quad
  T\leq
  \frac{1}{12M_1},
 \end{equation}
 and,
\begin{equation}\label{eq:4.15t}
 T
  \leq
  \left(
   8M_0
   \sum_{j=1}^3\frac{1}{|\vec{b}_\infty^{(j)}|-\Cr{const:4.2}}
  \right)^{-1}.
 \end{equation}
 Using the above estimates on time $T$, (\ref{eq:4.14t})-(\ref{eq:4.15t}) in (\ref{eq:4.13d}) we obtain
 that,
 \begin{equation*}
  d_X
  ((\Phi(\vec{\alpha}_1,\vec{a}_1), \Psi(\vec{\alpha}_1,\vec{a}_1)),
  (\Phi(\vec{\alpha}_2,\vec{a}_2), \Psi(\vec{\alpha}_2,\vec{a}_2))
  \leq
  \frac{3}{4}
  d_X
  ((\vec{\alpha}_1,\vec{a}_1),
  (\vec{\alpha}_2,\vec{a}_2)).
 \end{equation*}
 Therefore, by the contraction mapping principle, there is a fixed
 point
 $(\vec{\alpha},\vec{a})\in X_T$, such that
 \[
 \vec{\alpha}=\Phi(\vec{\alpha},\vec{a}),\quad
 \vec{a}=\Psi(\vec{\alpha},\vec{a}),
 \]
 which is a solution of the system of differential
 equations\eqref{eq:4.2}.

Moreover,  we obtain the following estimates:
 \begin{equation}
  \label{eq:4.23}
  \begin{split}
   \|\vec{\alpha}\|_{C([0,T])}
   &\leq
   2|\vec{\alpha}_0|,
   \quad
   \|\vec{a}-\vec{a}_\infty\|_{C([0,T])}
   \leq
   2|\vec{a}-\vec{a}_\infty|, \\
   T_{\text{max}}
   &\geq
   \min
  \left\{
   \frac{|\vec{\alpha}_0|}{4(M_1+8M_2|\vec{\alpha}_0|)\sum_{j=1}^3|\vec{b}^{(j)}_\infty|},\
   \frac{|\vec{a}_0-\vec{a}_\infty|}{3M_0}
  ,\
  \frac{1}{12M_1},\
  \frac{1}{8M_0\sum_{j=1}^3\frac{1}{|\vec{b}_\infty^{(j)}|-2|\vec{a}_0-\vec{a}_\infty|}}
 \right\},
  \end{split}
 \end{equation}
 where $T_{\text{max}}$ is a maximal existence time of the solution
 $(\vec{\alpha},\vec{a})$. 
\end{proof}

\begin{remark}\label{rem:4.6}
 Note, that once some a priori estimates for
 $\|\vec{\alpha}\|_{C([0,T])}$ and
 $\|\vec{a}-\vec{a}_\infty\|_{C([0,T])}$ are deduced, a global solution
 of \eqref{eq:4.2} can be obtained using the estimate of a maximal
 existence time  $T_{\text{max}}$.
\end{remark}

\section{A priori estimates}\label{sec:5}
We first derive the energy dissipation principle for the system \eqref{eq:4.2}. The
system does not depend on parametrization $s$, hence the energy of the system
\eqref{eq:4.2} is given by
\begin{equation}
  E(t)=
  \sum_{j=1}^{3}
  \sigma(\Delta\alpha^{(j)}(t))
  |\vec{b}^{(j)}(t)|.
\end{equation}

\begin{proposition}
 [Energy dissipation]
 \label{prop:5.1}
 Let $(\vec{\alpha},\vec{a})$ be a solution of \eqref{eq:4.2} for
 $0\leq t\leq T$. Then, for all $0<t \leq T$, we have the local
 dissipation equality,
 \begin{equation}
  \label{eq:5.1}
  E(t)
  +\int_0^t
  \left|
  \frac{d\vec{\alpha}}{dt}(\tau)
  \right|^2
  \,d\tau
  +\int_0^t
  \left|
  \frac{d\vec{a}}{dt}(\tau)
  \right|^2
  \,d\tau \\
  =
  E(0).
 \end{equation}
\end{proposition}

\begin{proof}
 [Proof of Proposition \ref{prop:5.1}]
 Let us first compute the rate of the dissipation of the energy of the
 system \eqref{eq:4.2} at time $t$,
 \begin{equation}
  \label{eq:5.2}
    \frac{d}{dt}
    E(t)
    =
    \sum_{j=1}^3
    \sigma_\MisOriAngle
    (\Delta\alpha^{(j)})
    \left(
    \frac{d\alpha^{(j-1)}}{dt}-\frac{d\alpha^{(j)}}{dt}
    \right)
    |\vec{b}^{(j)}| 
    +
    \sum_{j=1}^3
    \sigma
    (\Delta\alpha^{(j)})
    \frac{\vec{b}^{(j)}}{|\vec{b}^{(j)}|}
    \cdot
    \frac{d\vec{b}^{(j)}}{dt}.
 \end{equation}
 Since $(\vec{\alpha},\vec{a})$ is a solution of  the system \eqref{eq:4.2}, the
 right hand side of \eqref{eq:5.2} can be calculated as,
 \[
 \begin{split}
  \sum_{j=1}^3
  \sigma_\MisOriAngle(\Delta\alpha^{(j)})
  \left(
  \frac{d\alpha^{(j-1)}}{dt}-\frac{d\alpha^{(j)}}{dt}
  \right)
  |\vec{b}^{(j)}|
  &=
  \sum_{j=1}^3
  \Bigl(
  \sigma_\MisOriAngle(\Delta\alpha^{(j+1)})|\vec{b}^{(j+1)}|
  -
  \sigma_\MisOriAngle(\Delta\alpha^{(j)})|\vec{b}^{(j)}|
  \Bigr)
  \frac{d\alpha^{(j)}}{dt} \\
  &=
  -
  \sum_{j=1}^3
  \left|
  \frac{d\alpha^{(j)}}{dt}
  \right|^2,
 \end{split}
 \]
 and
 \[
 \sum_{j=1}^3
 \sigma(\Delta\alpha^{(j)})
 \frac{\vec{b}^{(j)}}{|\vec{b}^{(j)}|}
 \cdot
 \frac{d\vec{b}^{(j)}}{dt}
 =
 -
 \left|
 \frac{d\vec{a}}{dt}
 \right|^2.
 \]
Thus, we obtain the energy dissipation for the system,
\begin{equation}
  \label{eq:5.2a}
    \frac{d}{dt} E(t) = -
  \left|
  \frac{d\vec{\alpha}}{dt}
  \right|^2 -
 \left|
 \frac{d\vec{a}}{dt}
 \right|^2.
 \end{equation}
 Next, integrating \eqref{eq:5.2a} with respect to $t$, we have the local dissipation equality
 \eqref{eq:5.1}.
\end{proof}

From the energy dissipation and the assumption
\eqref{eq:2.Assumption1}, we obtain,
\begin{corollary}
 Let $(\vec{\alpha},\vec{a})$ be a solution of \eqref{eq:4.2} for
 $0\leq t\leq T$. Then, for all $0<t \leq T$,
 \begin{equation}
  \label{eq:5.6}
   |\vec{b}^{(j)}(t)|
   \leq
   \frac{1}{\Cr{const:2.Assumption1}}E(0).
 \end{equation}
\end{corollary}

\begin{proposition}
 [Maximum principle]
 \label{prop:5.2}
 Let $(\vec{\alpha},\vec{a})$ be a solution of the system \eqref{eq:4.2}  for
 $0\leq t\leq T$. Then, for all $0<t \leq T$, we have,
 \begin{equation}
  \label{eq:5.3}
   |\vec{\alpha}(t)|^2
   \leq
   |\vec{\alpha}_0|^2.
 \end{equation}
\end{proposition}

\begin{proof}
 [Proof of Proposition \ref{prop:5.2}]
 Multiplying the first equation of
 \eqref{eq:2.5} by  $\alpha^{(j)}$ and taking the sum for $j=1,2,3$, we obtain,
 \begin{equation}
  \begin{split}
   \frac{1}{2}\frac{d}{dt}|\vec{\alpha}(t)|^2
   &=
   -
   \sum_{j=1}^3
   \left(
   \sigma_\MisOriAngle(\Delta\alpha^{(j+1)})|\vec{b}^{(j+1)}|
   -
   \sigma_\MisOriAngle(\Delta\alpha^{(j)})|\vec{b}^{(j)}|
   \right)
   \alpha^{(j)} \\
   &=
   -
   \sum_{j=1}^3
   \left(
   \sigma_\MisOriAngle(\Delta\alpha^{(j)})|\vec{b}^{(j)}|
   \right)
   (\alpha^{(j-1)}-\alpha^{(j)})
   =
   -
   \sum_{j=1}^3
   \left(
   \sigma_\MisOriAngle(\Delta\alpha^{(j)})|\vec{b}^{(j)}|
   \right)
   \Delta\alpha^{(j)}.
  \end{split}
 \end{equation}
 Next, integrating with respect to $t$, and using the assumption
 \eqref{eq:2.Assumption2}, we obtain the result \eqref{eq:5.3}.
\end{proof}
\begin{proposition}
 [Misorientation estimates]
 \label{prop:5.3}
 Let $(\vec{\alpha},\vec{a})$ be a solution of \eqref{eq:4.2} for
 $0\leq t\leq T$. Then, for all $0<t\leq T$, we have the following
 estimate for the misorientation,
 \begin{equation}
  \label{eq:5.4}
   \sum_{j=1}^3
    (\Delta\alpha^{(j)}(t))^2
   \leq
   \sum_{j=1}^3
   (\Delta\alpha^{(j)}(0))^2.
 \end{equation}
\end{proposition}

\begin{proof}
 [Proof of Proposition \ref{prop:5.3}]
 We take a derivative on the misorientation $\Delta\alpha^{(j)}$ with
 respect to $t$,
 \begin{equation}\label{eq:5.5a}
  \frac{d}{dt}
   \Delta\alpha^{(j)}
   =
   \alpha_t^{(j-1)}
   -
   \alpha_t^{(j)}
   =
   -
   2\sigma_\MisOriAngle(\Delta\alpha^{(j)})|\vec{b}^{(j)}|
   +
   \sigma_\MisOriAngle(\Delta\alpha^{(j-1)})|\vec{b}^{(j-1)}|
   +
   \sigma_\MisOriAngle(\Delta\alpha^{(j+1)})|\vec{b}^{(j+1)}|.
 \end{equation}
 Next we multiply (\ref{eq:5.5a}) by $\Delta\alpha^{(j)}$ and take the sum for
 $j=1,2,3$, we obtain,
 \begin{equation}
  \label{eq:5.5}
   \begin{split}
    &\quad
    \frac12\frac{d}{dt}
    \left(
    \sum_{j=1}^3
    (\Delta\alpha^{(j)}(t))^2
    \right) \\
    &=
    \sum_{j=1}^3
    \left(
    -
    2\sigma_\MisOriAngle(\Delta\alpha^{(j)})|\vec{b}^{(j)}|
    +
    \sigma_\MisOriAngle(\Delta\alpha^{(j-1)})|\vec{b}^{(j-1)}|
    +
    \sigma_\MisOriAngle(\Delta\alpha^{(j+1)})|\vec{b}^{(j+1)}|
    \right)
    \Delta\alpha^{(j)} \\
    &=
    \sum_{j=1}^3
    \sigma_\MisOriAngle(\Delta\alpha^{(j)})|\vec{b}^{(j)}|
    \left(
    -
    2\Delta\alpha^{(j)}
    +
    \Delta\alpha^{(j+1)}
    +
    \Delta\alpha^{(j-1)}
    \right) \\
    &=
    -3
    \sum_{j=1}^3
    \sigma_\MisOriAngle(\Delta\alpha^{(j)})|\vec{b}^{(j)}|
    \Delta\alpha^{(j)}.
   \end{split}
 \end{equation}

 Next, integrating \eqref{eq:5.5} with respect to $t$, we obtain,
 \begin{equation}
  \label{eq:5.4a}
   \sum_{j=1}^3
   (\Delta\alpha^{(j)}(t))^2
   +
   6
   \sum_{j=1}^3
   \int_0^t
   \sigma_\MisOriAngle(\Delta\alpha^{(j)})|\vec{b}^{(j)}|
   \Delta\alpha^{(j)}
   \, d\tau
   =
   \sum_{j=1}^3
   (\Delta\alpha^{(j)}(0))^2.
 \end{equation}
 Similar to the Proposition \ref{prop:5.2}, we use the convexity
 assumption \eqref{eq:2.Assumption2}, hence we obtain final result
 \eqref{eq:5.4}.
\end{proof}

\begin{remark}\label{rem:5.5}
 1. Usually, the misorientations are assumed to be bounded by
some constant, hence the orientations are also bounded. In 2D case, it
is reasonable to consider
misorientations in the interval between $-\pi/4$ and $\pi/4$ (see, for example, \cite{DK:gbphysrev}). In this
case,  one can consider the orientations within $-\pi/8$ and
$\pi/8$.\\
 2. Proposition \ref{prop:5.3} guarantees consistency for
 misorientations, which is $-\pi/4\leq\Delta\alpha^{(j)}(t)\leq\pi/4$,
 see work, for example,
 (\cite{DK:BEEEKT,DK:gbphysrev,MR2772123,MR3729587,MR3333842}) for
 bounds on misorientation in 2D. Indeed, if the $l^2$ sum of three initial
 misorientations is bounded by $\pi/4$, that is
 $(\sum_{j=1}^3(\Delta\alpha^{(j)}(0))^2)^\frac12\leq\pi/4$, then the
 magnitude of the misorientation has the same bounds
 $|\Delta\alpha^{(j)}(t)|<\pi/4$ for $t>0$.
\end{remark}

\section{Uniqueness and continuous dependence}
\label{sec:6}

In this section, we show uniqueness and continuous dependence on the
initial data of the solution of the system \eqref{eq:4.2}.

\begin{lemma}
 \label{lem:6.1}
 For $\vec{x}^{(1)}$, $\vec{x}^{(2)}$, $\vec{x}^{(3)}\in \R^2$,
 $\vec{a}_{01}, \vec{a}_{02}\in\R^2$, and $\vec{\alpha}_{01},
 \vec{\alpha}_{02}\in\R^3$, assume that $(\vec{\alpha}_1 (t),\vec{a}_1(t))$ and
 $(\vec{\alpha}_2(t),\vec{a}_2(t))$ are classical solutions of
 \eqref{eq:4.2}  on time interval $0\leq t\leq
 T$,
 associated with the given initial data $(\vec{\alpha}_{01},\vec{a}_{01})$ and
 $(\vec{\alpha}_{02},\vec{a}_{02})$, respectively. Next, assume that there exists a constant $\Cl{const:6.1}>0$ such that
 $|\vec{b}^{(j)}_k(t)|\geq \Cr{const:6.1}$ for $0\leq t\leq T$,
 $j=1,2,3$ and $k=1,2$. Here, $\vec{b}^{(j)}_k(t):=\vec{x}^{(j)}
 -\vec{a}_k (t)$, $j=1,2,3$ and $k=1,2$. Then,
 \begin{equation}
  \label{eq:6.3}
  \frac{d}{dt}
   (
   |\vec{\alpha}_1-\vec{\alpha}_2|^2
   +
   |\vec{a}_1-\vec{a}_2|^2
   )
   \leq
   \Cl{const:6.6}(|\vec{\alpha}_1-\vec{\alpha}_2|^2
   +
   |\vec{a}_1-\vec{a}_2|^2)
 \end{equation}
 holds, where $\Cr{const:6.6}>0$ is a positive constant that is
 independent of $(\vec{\alpha}_1,\vec{a}_1)$ and
 $(\vec{\alpha}_2,\vec{a}_2)$.
\end{lemma}

\begin{remark}
 To be precise, the constant
 $\Cr{const:6.6}>0$, in Lemma \ref{lem:6.1}, depends on $\Cr{const:2.Assumption1}$,
 $\Cr{const:6.1}$, $E_1(0)=\sum_{j=1}^{3}
 \sigma(\Delta\alpha_1^{(j)}(0)) |\vec{b}^{(j)}(0)|$, and
 \begin{equation}
  \label{eq:6.6}
  M:=\sup
   \left\{
    |\sigma(\MisOriAngle)|
    +
    |\sigma_\MisOriAngle(\MisOriAngle)|
    +
    \frac{|\sigma_\MisOriAngle(\MisOriAngle_1)-\sigma_\MisOriAngle(\MisOriAngle_2)|}{|\MisOriAngle_1-\MisOriAngle_2|}:
    |\MisOriAngle|,|\MisOriAngle_1|, |\MisOriAngle_2|\leq\max_{k=1,2}(\sum_{j=1}^3|\Delta\alpha_k^{(j)}(0)|^2)^\frac12
   \right\}.
 \end{equation}
\end{remark}

\begin{proof}
 [Proof of Lemma \ref{lem:6.1}]
 %
 Using the equation \eqref{eq:4.2}, we have that,
 \[
  \begin{split}
  &\quad\frac{d}{dt}
  (\alpha^{(j)}_1-\alpha^{(j)}_2) \\
  &=
  -
  \left(
  \sigma_\MisOriAngle(\Delta\alpha^{(j+1)}_1)|\vec{b}_1^{(j+1)}|
  -\sigma_\MisOriAngle(\Delta\alpha^{(j+1)}_2)|\vec{b}_2^{(j+1)}|
  \right)
  +
  \left(
  \sigma_\MisOriAngle(\Delta\alpha^{(j)}_1)|\vec{b}_1^{(j)}|
  -\sigma_\MisOriAngle(\Delta\alpha^{(j)}_2)|\vec{b}_2^{(j)}|
  \right),
 \end{split}
\]
 and, hence, multiplying by $\alpha^{(j)}_1-\alpha^{(j)}_2$ and taking
 the sum for $j=1,2,3$, we obtain,
 \begin{equation}
  \label{eq:6.1a}
   \begin{split}
    \frac12\frac{d}{dt}
    |\vec{\alpha}_1-\vec{\alpha}_2|^2
    &=
    -
    \sum_{j=1}^3
    \left(
    \sigma_\MisOriAngle(\Delta\alpha^{(j+1)}_1)|\vec{b}_1^{(j+1)}|
    -\sigma_\MisOriAngle(\Delta\alpha^{(j+1)}_2)|\vec{b}_2^{(j+1)}|
    \right)
    (\alpha_1^{(j)}-\alpha_2^{(j)}) \\
    &\qquad
    +
    \sum_{j=1}^3
    \left(
    \sigma_\MisOriAngle(\Delta\alpha^{(j)}_1)|\vec{b}_1^{(j)}|
    -\sigma_\MisOriAngle(\Delta\alpha^{(j)}_2)|\vec{b}_2^{(j)}|
    \right)
    (\alpha_1^{(j)}-\alpha_2^{(j)}). \\
   \end{split}
 \end{equation}
 The estimate for the first term on the right hand side
 of (\ref{eq:6.1a}) is obtained using Lipschitz continuity of
 $\sigma_\MisOriAngle$, \eqref{eq:5.6}, and \eqref{eq:5.3},
 \begin{equation*}
  \begin{split}
   &\quad
   \left(
   \sigma_\MisOriAngle(\Delta\alpha^{(j+1)}_1)|\vec{b}_1^{(j+1)}|
   -\sigma_\MisOriAngle(\Delta\alpha^{(j+1)}_2)|\vec{b}_2^{(j+1)}|
   \right)
   (\alpha_1^{(j)}-\alpha_2^{(j)}) \\
   &\leq
   |\vec{\alpha}_1-\vec{\alpha}_2|
   \left(
   \left|
   \sigma_\MisOriAngle(\Delta\alpha^{(j+1)}_1)
   -
   \sigma_\MisOriAngle(\Delta\alpha^{(j+1)}_2)
   \right|
   |\vec{b}_1^{(j+1)}|
   +
   \left|
   \sigma_\MisOriAngle(\Delta\alpha^{(j+1)}_2)
   \right|
   |\vec{b}_1^{(j+1)}-\vec{b}_2^{(j+1)}|
   \right) \\
   &\leq
   |\vec{\alpha}_1-\vec{\alpha}_2|
   \left(
   \frac{M}{\Cr{const:2.Assumption1}}E_1(0)
   \left|
   \Delta\alpha^{(j+1)}_1
   -
   \Delta\alpha^{(j+1)}_2
   \right|
   +
   M
   |\vec{a}_1-\vec{a}_2|
   \right) \\
   &\leq
   \frac{2M}{\Cr{const:2.Assumption1}}E_1(0)
   |\vec{\alpha}_1-\vec{\alpha}_2|^2
   +
   M
   |\vec{\alpha}_1-\vec{\alpha}_2|
   |\vec{a}_1-\vec{a}_2|,
  \end{split}
 \end{equation*}
 where the constant $M>0$ is given by \eqref{eq:6.6}
 and $E_1(0)=\sum_{j=1}^{3} \sigma(\Delta\alpha_1^{(j)}(0))
 |\vec{b}^{(j)}(0)|$. The second term on the right hand side of
 \eqref{eq:6.1a} can be handled the same way. Next, using
 the Young's inequality for the estimate in the right-hand side of
 (\ref{eq:6.1a}), we deduce,
 \begin{equation}
  \label{eq:6.1}
   \frac{d}{dt}
   |\vec{\alpha}_1-\vec{\alpha}_2|^2
   \leq
   6M
   \left(
    \frac{4}{\Cr{const:2.Assumption1}}
    E(0)
    +
    1
   \right)
   |\vec{\alpha}_1-\vec{\alpha}_2|^2
   +
   6M
   |\vec{a}_1-\vec{a}_2|^2
   .
 \end{equation}
 Similarly, from the equation \eqref{eq:4.2}, we have that,
 \[
 \begin{split}
  \frac{d}{dt}
  (\vec{a}_1-\vec{a}_2)
  &=
  \sum_{j=1}^3
  \sigma(\Delta\alpha_1^{(j)})
  \frac{\vec{b}_1^{(j)}}{|\vec{b}_1^{(j)}|}
  -
  \sigma(\Delta\alpha_2^{(j)})
  \frac{\vec{b}_2^{(j)}}{|\vec{b}_2^{(j)}|}  \\
  &=
  \sum_{j=1}^3
  \left(
  \sigma(\Delta\alpha_1^{(j)})
  -
  \sigma(\Delta\alpha_2^{(j)})
  \right)
  \frac{\vec{b}_1^{(j)}}{|\vec{b}_1^{(j)}|}
  +
  \sum_{j=1}^3
  \sigma(\Delta\alpha_2^{(j)})
  \left(
  \frac{\vec{b}_1^{(j)}}{|\vec{b}_1^{(j)}|}
  -
  \frac{\vec{b}_2^{(j)}}{|\vec{b}_2^{(j)}|}
  \right).
 \end{split}
 \]
 Hence, we obtain,
 \begin{equation}
  \label{eq:6.1b}
  \begin{split}
   &\quad \frac12\frac{d}{dt}
   |\vec{a}_1-\vec{a}_2|^2 \\
   &=
   \sum_{j=1}^3
   \left(
   \sigma(\Delta\alpha_1^{(j)})
   -
   \sigma(\Delta\alpha_2^{(j)})
   \right)
   \left(
   \frac{\vec{b}_1^{(j)}}{|\vec{b}_1^{(j)}|}
   \cdot
   (\vec{a}_1-\vec{a}_2)
   \right)
   +
   \sum_{j=1}^3
   \sigma(\Delta\alpha_2^{(j)})
   \left(
   \frac{\vec{b}_1^{(j)}}{|\vec{b}_1^{(j)}|}
    -
   \frac{\vec{b}_2^{(j)}}{|\vec{b}_2^{(j)}|}
   \right)
   \cdot
   \left(
   \vec{a}_1-\vec{a}_2
   \right) \\
   &\leq
   \sum_{j=1}^3
   M
   |\Delta\alpha_1^{(j)}-\Delta\alpha_2^{(j)}|
   \left|
   \vec{a}_1-\vec{a}_2
   \right|
   +
   \sum_{j=1}^3
   \sigma(\Delta\alpha_2^{(j)})
   \left|
   \frac{\vec{b}_1^{(j)}}{|\vec{b}_1^{(j)}|}
    -
   \frac{\vec{b}_2^{(j)}}{|\vec{b}_2^{(j)}|}
   \right|
   \left|
   \vec{a}_1-\vec{a}_2
   \right| \\
   &\leq
   6M
   |\vec{\alpha}_1-\vec{\alpha}_2|
   \left|
   \vec{a}_1-\vec{a}_2
   \right|
   +
   \sum_{j=1}^3
   \sigma(\Delta\alpha_2^{(j)})
   \left|
   \frac{\vec{b}_1^{(j)}}{|\vec{b}_1^{(j)}|}
    -
   \frac{\vec{b}_2^{(j)}}{|\vec{b}_2^{(j)}|}
   \right|
   \left|
   \vec{a}_1-\vec{a}_2
   \right|.
  \end{split}
 \end{equation}
 Next, let us estimate the second term on the right-hand side of the
 (\ref{eq:6.1b}).
 Applying \eqref{eq:4.12}, and using that $|\vec{b}^{(j)}_k(t)|\geq
 \Cr{const:6.1}$, $\vec{b}^{(j)}_k(t)=\vec{x}^{(j)} -\vec{a}_k (t)$ for
 $j=1,2,3$ and $k=1,2$, we have that,
 \begin{equation*}
  \sigma(\Delta\alpha_2^{(j)})
   \left|
    \frac{\vec{b}_1^{(j)}}{|\vec{b}_1^{(j)}|}
    -
    \frac{\vec{b}_2^{(j)}}{|\vec{b}_2^{(j)}|}
   \right|
   \left|
    \vec{a}_1-\vec{a}_2
   \right|
   =
   \frac{2}{|\vec{b}_1^{(j)}|}
   \sigma(\Delta\alpha_2^{(j)})
   \left|
    \vec{b}_1^{(j)}
    -
    \vec{b}_2^{(j)}
   \right|
   \left|
    \vec{a}_1-\vec{a}_2
   \right|
   \leq
   \frac{2M}{\Cr{const:6.1}}
   \left|
    \vec{a}_1-\vec{a}_2
   \right|^2.
 \end{equation*}
 Hence, we have that,
 \begin{equation}
  \label{eq:6.2a}
   \frac{d}{dt}
   |\vec{a}_1-\vec{a}_2|^2
   \leq
   6M
   |\vec{\alpha}_1-\vec{\alpha}_2|^2
   +
   6M
   \left(
    \frac{2}{\Cr{const:6.1}}
    +
    1
   \right)
   \left|
    \vec{a}_1-\vec{a}_2
   \right|^2.
 \end{equation}
 Therefore, by \eqref{eq:6.1} and \eqref{eq:6.2a}, we have,
 \[
 \frac{d}{dt}
 (
 |\vec{\alpha}_1-\vec{\alpha}_2|^2
 +
 |\vec{a}_1-\vec{a}_2|^2
 )
 \leq
 \Cl{const:6.4}|\vec{\alpha}_1-\vec{\alpha}_2|^2
 +
 \Cl{const:6.5}|\vec{a}_1-\vec{a}_2|^2,
 \]
 where,
 \[
 \Cr{const:6.4}
 :=
 12M
 \left(
 \frac{2}{\Cr{const:2.Assumption1}}E(0)+1\right)
 ,\quad
 \Cr{const:6.5}
 :=
 12M
 \left(
 \frac{1}{\Cr{const:6.1}}
 +
 1
 \right).
 \]
\end{proof}

By the neighboring inequality, we can now show uniqueness of the
classical solution to the system \eqref{eq:4.2}.

\begin{theorem}
 [Uniqueness] \label{thm:6.1}
 Consider $\vec{x}^{(1)}$, $\vec{x}^{(2)}$, $\vec{x}^{(3)}\in \R^2$, and
 initial data
 $\vec{a}_0\in\R^2$ and $\vec{\alpha}_0\in\R^3$. Assume also, that
 there exists a constant $\Cl{const:6.2}>0$, such that
 $|\vec{b}^{(j)}_k(t)|\geq \Cr{const:6.2}$ for $0\leq t\leq T$, $j=1,2,3$
 and $k=1,2$. Then, there exists a unique classical solution
 $(\vec{\alpha}(t),\vec{a}(t))$ $0\leq t\leq
 T$ of the system \eqref{eq:4.2}.
\end{theorem}

Note that,  $\Cr{const:6.6}$ stays bounded when
$(\vec{\alpha}_{01},\vec{a}_{01}) \rightarrow
(\vec{\alpha}_{02},\vec{a}_{02})$. Thus, we obtain,

\begin{theorem}
 [Continuous dependence on the initial data]
 \label{thm:6.2}
 For $\vec{x}^{(1)}$, $\vec{x}^{(2)}$, $\vec{x}^{(3)}\in \R^2$,
 $\vec{a}_{01}, \vec{a}_{02}\in\R^2$ and $\vec{\alpha}_{01},
 \vec{\alpha}_{02}\in\R^3$, let $(\vec{\alpha}_1,\vec{a}_1)$ and
 $(\vec{\alpha}_2,\vec{a}_2)$ be two classical solutions of the system
 \eqref{eq:4.2} on $0\leq t\leq
 T$,
 associated with the given initial data $(\vec{\alpha}_{01},\vec{a}_{01})$ and
 $(\vec{\alpha}_{02},\vec{a}_{02})$,  respectively. Assume, that there exists a constant $\Cl{const:6.3}>0$, such that
 $|\vec{b}^{(j)}_k(t)|\geq \Cr{const:6.3}$ for $0\leq t\leq T$,
 $j=1,2,3$ and $k=1,2$. Then,
 \begin{equation}
  \label{eq:6.4}
   |\vec{\alpha}_1-\vec{\alpha}_2|^2
   +
   |\vec{a}_1-\vec{a}_2|^2
   \leq
   e^{\Cr{const:6.6}t}
   (|\vec{\alpha}_{01}-\vec{\alpha}_{02}|^2
   +
   |\vec{a}_{01}-\vec{a}_{02}|^2)
 \end{equation}
 holds, where $\Cr{const:6.6}>0$ is a positive constant given in Lemma
 \ref{lem:6.1}. In particular, continuous dependence on the
 initial data holds, namely,
 \[
 \|\vec{\alpha}_1-\vec{\alpha}_2\|_{C([0,T])}
 +
 \|\vec{a}_1-\vec{a}_2\|_{C([0,T])}
 \rightarrow0
 \]
 as $(\vec{\alpha}_{01},\vec{a}_{01}) \rightarrow
 (\vec{\alpha}_{02},\vec{a}_{02})$.
\end{theorem}

\section{Evolution of grain boundary network}\label{sec:7}

In this section, we extend the results obtained above for a system with a
single junction to a network of grains that have lattice
orientations $\{\alpha^{(k)}\}_{k=1}^{N^{\text{SG}}}$, grain boundaries $\{\Gamma_t^{(j)}\}_{j=1}^{N^{\text{GB}}}$ and the triple junctions $\{\vec{a}^{(l)}\}_{l=1}^{N^{\text{TJ}}}$. We
identify the lattice $\alpha^{(k)}$ with the single grain
$k$. Hence, the grain boundary energy of the entire network is defined
now as,
\begin{equation}
 E(t)
  =
  \sum_{j=1}^{N^{\text{GB}}}
  \int_{\Gamma_t^{(j)}}
  \sigma
  (
  \vec{n}^{(j)},
  \Delta^{(j)}\alpha
  )
  \,d\mathscr{H}^1,
\end{equation}
where $\Delta^{(j)}\alpha$ is a difference between the lattice orientions
of the two grains that share the same grain boundary $\Gamma^{(j)}$.
The difference $\Delta^{(j)}\alpha$ is called a misorientation of the grain boundary
$\Gamma^{(j)}$. Next, using
the same argument as in Section \ref{sec:2} for a system with a single
triple junction, we obtain similar expression for the dissipation rate of the
energy of the grain boundary network,
\begin{equation}
 \begin{split}
  \frac{d}{dt}E(t)
  &=
  -
  \sum_{j=1}^{N^{\text{GB}}}
  \int_{\Gamma^{(j)}}
  \frac{d}{ds}
  \vec{T}^{(j)}
  \,d\mathscr{H}^1
  +
  \sum_{k=1}^{N^{\text{SG}}}
  \frac{\partial E}{\partial \alpha^{(k)}}
  \frac{d\alpha^{(k)}}{dt}
  -
  \sum_{l=1}^{N^\text{TJ}}
  \sum_{\vec{a}^{(l)}\in\Gamma^{(j)}_t}
  \vec{T}^{(j)}
  \cdot
  \frac{d\vec{a}^{(l)}}{dt}.
 \end{split}
\end{equation}
Here,
\begin{equation}
 \vec{T}^{(j)}
  =
  \sigma_\NorAngle^{(j)}\hat{\vec{n}}^{(j)}
  +
  \sigma^{(j)}\hat{\vec{b}}^{(j)},
\end{equation}
and $\vec{a}^{(l)}$ denotes the triple junction where three
grain boundaries meet (we assume in our model that only triple
junctions are stable). Note that, the line tension vector $\vec{T}^{(j)}$
points toward an inward direction of the grain boundary at the triple junction $\vec{a}^{(l)}$.


Next, similar to  Section \ref{sec:2}, we obtain the following
system of differential equations to ensure that the entire system is dissipative:
\begin{equation}
 \label{eq:7.4}
 \begin{aligned}
  v^{(j)}
  &=
  \mu^{(j)}
  \frac{d}{ds}\vec{T}^{(j)}\cdot\hat{\vec{n}}^{(j)},&\quad
  &j=1,\ldots,N^{\text{GB}}, \\
  \frac{d\alpha^{(k)}}{dt}
  &=
  -\gamma
  \frac{\delta E}{\delta \alpha^{(k)}},&\quad
  &k=1,\ldots,N^{\text{SG}}, \\
  \frac{d\vec{a}^{(l)}}{dt}
  &=
  \eta
  \sum_{\vec{a}^{(l)}\in\Gamma_t^{(j)}}
  \vec{T}^{(j)},&\quad
  &l=1,\ldots,N^{\text{TJ}},
 \end{aligned}
\end{equation}
where $\mu^{(j)},\gamma,\eta>0$ are positive constants.
For simplicity of the calculations below, we further assume that
the energy density $\sigma(\vec{n},\MisOriAngle)$ is an even function with
respect to the misorientation $\MisOriAngle= \Delta^{(j)}\alpha$, that is, the misorientation
effects are symmetric with respect to the difference between the lattice
orientations.
%
%
For the two grains $k_1$ and $k_2$ with orientations $\alpha^{(k_1)}$
and $\alpha^{(k_2)}$, respectively, we introduce notation that will
be helpful for calculations below, $\Gamma^{(j)}:=\Gamma^{(j(k_1,k_2))}$
a grain boundary which is formed by grains $k_1$ and $k_2$ (See Figure
\ref{fig:2}). We also assume, that if grains $k_1$ and $k_2$ have no
common interface/grain boundary, then we
just set $\Gamma^{(j(k_1,k_2))}=\emptyset$. Then,
\begin{equation}
 \label{eq:7.2}
 \frac{\delta E}{\delta \alpha^{(k)}}
  =
  \sum_{\substack{k'=1, \\ k'\neq k}}^{N^{\text{SG}}}
  \int_{\Gamma^{(j(k,k'))}_t}
  \sigma_\MisOriAngle(\vec{n}^{(j(k,k'))},\alpha^{(k)}-\alpha^{(k')})
  \,d\mathscr{H}^1.
\end{equation}

 \begin{figure}
  \begin{minipage}{0.45\textwidth}
   \includegraphics[height=5cm]{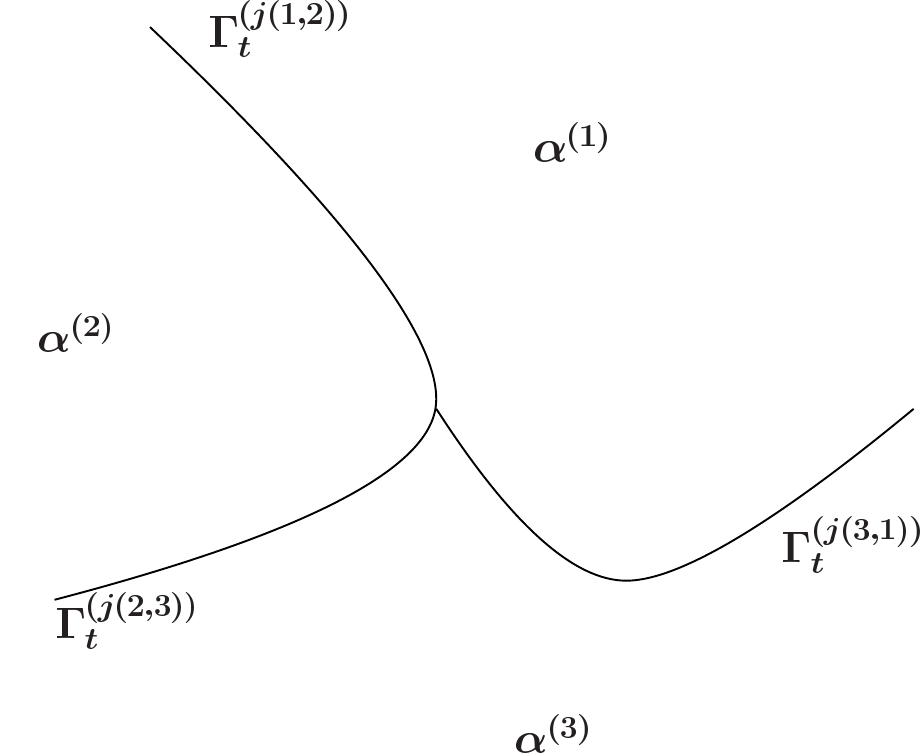}
   \caption{Example of $\Gamma^{(j(k_1,k_2))}$}
   \label{fig:2}
  \end{minipage}
 \end{figure}

We let $\mu^{(j)}\rightarrow\infty$, $\gamma=\eta=1$ and as before, we
consider surface tension
(\ref{eq:2.Assumption1})-(\ref{eq:2.Assumption3}) without normal
dependence.
Then, the problem \eqref{eq:7.4} is turned into,
\begin{equation}
 \label{eq:7.7}
  \begin{aligned}
   \Gamma_t^{(j)} &\text{is a line segment between some}\
   \vec{a}^{(l_{j,1})}\ \text{and}\ \vec{a}^{(l_{j,2})},&\quad
   &j=1,\ldots,N^{\text{GB}}, \\
   \frac{d\alpha^{(k)}}{dt}
   &=
   -
   \sum_{\substack{k'=1, \\ k'\neq k}}^{N^{\text{SG}}}
   |\Gamma^{(j(k,k'))}_t|
   \sigma_\MisOriAngle(\alpha^{(k)}-\alpha^{(k')})
   ,&\quad
   &k=1,\ldots,N^{\text{SG}}, \\
   \frac{d\vec{a}^{(l)}}{dt}
   &=
   \sum_{\vec{a}^{(l)}\in\Gamma_t^{(j)}}
  \vec{T}^{(j)},&\quad
   &l=1,\ldots,N^{\text{TJ}},
  \end{aligned}
\end{equation}

Due to the convexity assumption \eqref{eq:2.Assumption2}, we obtain the
maximum principle for $\alpha^{(k)}$. In fact, for a fixed
$j=1,\ldots,N^{\text{GB}}$, there are only two grains
$k_{j_1},k_{j_2}\in\{1,\ldots,N^{\text{SG}}\}$ such that $\Gamma^{(j)}$
is formed between grains $k_{j_1}$ and $k_{j_2}$. Using this fact, we
find that,
\begin{equation}
 \label{eq:7.20}
  \begin{split}
   &\quad
  \sum_{k=1}^{N^{\text{SG}}}
  \sum_{\substack{k'=1, \\ k'\neq k}}^{N^{\text{SG}}}
  |\Gamma^{(j(k,k'))}_t|
  \sigma_\MisOriAngle
  (\alpha^{(k)}-\alpha^{(k')})\alpha^{(k)} \\
  &=
  \sum_{j=1}^{N^{\text{GB}}}
  |\Gamma^{(j)}_t|
  \left(
  \sigma_\MisOriAngle
  (
  \alpha^{(k_{j,1})}
  -
  \alpha^{(k_{j,2})}
  )
  \alpha^{(k_{j,1})}
  +
  \sigma_\MisOriAngle
  (
  \alpha^{(k_{j,2})}
  -
  \alpha^{(k_{j,1})}
  )
  \alpha^{(k_{j,2})}
  \right)\\
  &=
  \sum_{j=1}^{N^{\text{GB}}}
  |\Gamma^{(j)}_t|
  \sigma_\MisOriAngle(
  \alpha^{(k_{j,1})}
  -
  \alpha^{(k_{j,2})}
  )
  \left(
  \alpha^{(k_{j,1})}
  -
  \alpha^{(k_{j,2})}
  \right)\geq0.
  \end{split}
\end{equation}
Thus, we can proceed now using the same arguments as in Sections \ref{sec:4}-\ref{sec:6}. To show the existence of solution of \eqref{eq:7.7}, we
integrate \eqref{eq:7.7} and rewrite in the form of integral
equations. After that, we can make a
contraction mapping argument as it was done in Section \ref{sec:4} for
a single triple junction. The key ingredient in this approach is to show
a priori lower bounds for the distance of two triple junctions,
similar to
Lemma \ref{lem:4.3}. If an initial grain boundary network is
sufficiently close to some equilibrium state, then any triple
junction is close to its associated initial position (moreover, no
critical events happen during short enough time interval). Thus,
we can obtain a priori lower bounds for the distance between the two triple junctions. 
The uniqueness and continuous dependence on the initial
data can be obtained in a similar way as discussed in
Theorem \ref{thm:6.1} and \ref{thm:6.2}.
 Indeed, as in
Remark \ref{rem:2.4}, the convexity assumption \eqref{eq:2.Assumption2}
is not needed to show the uniqueness and continuous
dependence. Nevertheless, the convexity assumption
\eqref{eq:2.Assumption2} and its consequence, the result
\eqref{eq:7.20} are important if one would like
to guarantee the maximum principle type  result for the
orientations, similar to
Proposition \ref{prop:5.2}. 
Therefore, we obtain,
\begin{theorem}
 In a grain boundary network with lattice orientations, if
 triple junctions at the initial state are sufficiently close to triple junctions
 at the equilibrium state, then
 the problem \eqref{eq:7.7} has a unique time local solution 
 and the magnitude of the orientation of each grain is bounded by the
 $l^2$ sum of the initial orientations of the grains in the network,
 that is,
 $(\alpha^{(k')}(t))^2\leq\sum_{k}(\alpha^{(k)}(0))^2$ for $t>0$.
\end{theorem}

\begin{remark}
\par Note, that the proposed model of dynamic orientations (\ref{eq:7.4}) (and, hence,
dynamic misorientations, (\ref{eq:7.7}), or Langevin type
equation if critical events/grain boundaries disappearance events are taken into account) is reminiscent of the recently developed
theory for the grain boundary character distribution (GBCD)
\cite{DK:BEEEKT,DK:gbphysrev,MR2772123,MR3729587},  which suggests that the
evolution of the GBCD satisfies a Fokker-Planck Equation (GBCD is
an empirical distribution of the relative length (in 2D) or area (in
3D) of interface with a given lattice misorientation and normal). More
details will be presented in future studies.
\par Large time asymptotic analysis of the model proposed in the
current work will be
presented in the forthcoming paper.
\end{remark}

\section*{Acknowledgments}
The authors are grateful to David Kinderlehrer for the fruitful
discussions, inspiration and motivation of the work. The authors would
also like to thank Katayun Barmak for insightful discussions related to
various aspects of the orientations/misorientations and the grain boundary
energy density. The authors are grateful to the anonymous referees for their valuable remarks and
questions, which led to significant improvements of the manuscript.  Yekaterina
Epshteyn and Masashi Mizuno acknowledge partial support of Simons
Foundation Grant No. 415673, Yekaterina
Epshteyn also acknowledges partial support of NSF DMS-1905463,  Masashi Mizuno
also acknowledges partial support of JSPS KAKENHI Grant No. 18K13446, Chun Liu acknowledges partial support of
NSF DMS-1759535 and NSF DMS-1759536.
Masashi Mizuno also would like to thank Penn State University, Illinois
Institute of Technology and the University of Utah for the hospitality
during his visit.

\end{document}